\documentclass[smallextended]{svjour3}       
\pdfoutput=1

\smartqed  

\usepackage{fullpage,graphicx,times}

\usepackage[T1]{fontenc}
\usepackage[utf8]{inputenc}
\usepackage{amssymb,amsmath}

\usepackage[usenames]{color}

\newcommand{\EE}{E}

\newcommand\RR{\mathbb{R}}
\newcommand\ZZ{\mathbb{Z}}
\newcommand\tfin{\tau_{\textup{fin}}}

\DeclareMathOperator{\meas}{meas}
\DeclareMathOperator{\sign}{sign}
\DeclareMathOperator{\Span}{span}

\renewcommand\makeheadbox{{%
\noindent Published in \emph{Acta Applicandae Mathematicae}, vol. 135, pp 47-80 
, feb. 2015
, \hfill doi: 10.1007/s10440-014-9948-2. 
\\
{\it 
Some misprints from the journal paper have been corrected in the
present version.} }} 

\renewcommand\appendix{\section*{APPENDIX}\vspace{-.5\baselineskip}
  \setcounter{section}{0}%
  \setcounter{subsection}{0}%
  \renewcommand\thesection{\Alph{section}}}

\begin{document}

\title{Time Versus Energy in the Averaged Optimal Coplanar Kepler Transfer towards Circular Orbits 
    \thanks{The second author was partially supported by \textit{Thales Alenia Space} and \textit{région Provence Alpes
        Côte d'Azur}
     }
}

\titlerunning{Averaged Optimal Coplanar Kepler Transfer} 

\author{
Bernard Bonnard
  \and 
Helen C. Henninger
  \and \\
Jana N{\v e}mcov\'a
  \and 
Jean-Baptiste Pomet
}

\authorrunning{B. Bonnard,
   H. Henninger,
   J. N{\v e}mcov\'a,
   J.-B. Pomet}

\institute{
              B. Bonnard \at
              Institut de Mathématiques de Bourgogne, Université de Bourgogne, \\9 avenue Alain Savary, 21078 Dijon,
              France.  \\\email{bernard.bonnard@u-bourgogne.fr} 
              \\ 
              On leave to: team McTAO, Inria Sophia Antipolis Méditerrannée.
           \and
              H. Henninger \at 
              team McTAO, Inria Sophia Antipolis Méditerrannée, \\ 2004 rte des lucioles, B.P. 92, 06902 Sophia Antipolis cedex, France.\\
              \email{helen-clare.henninger@inria.fr}
           \and
              J. N{\v e}mcov\'a \at
Department of Mathematics,
Institute of Chemical Technology,\\
Technick{\'a} 5,
166 28 Prague 6,
Czech Republic. \\\email{jana.nemcova@vscht.cz}
           \and
              J.-B. Pomet \at
              team McTAO, Inria Sophia Antipolis Méditerrannée, \\ 2004 rte des lucioles, B.P. 92, 06902 Sophia Antipolis cedex, France.\\
              \email{jean-baptiste.pomet@inria.fr}
}

\date{Received: date / Accepted: date}

\maketitle

\begin{abstract}
  This article makes a study of the averaged optimal coplanar transfer towards circular orbits.
  Our objective is to compare this problem when the  cost minimized is  transfer time to the same problem when the cost minimized is energy consumption.
  While the minimum energy case leads to the analysis of a $2D-$ Riemannian metric using
  the standard tools of Riemannian geometry, the minimum time case is
  associated with a Finsler metric which is not smooth. Nevertheless a qualitative analysis of the geodesic flow is
  given in this article to describe the optimal transfers of the time minimal case. 
\end{abstract}

\keywords{
  Averaging, Optimal control, Low thrust orbit transfer, Geodesic convexity, Riemann-Finsler Geometry
}

\section{Introduction}

We consider the controlled Kepler equation describing orbital transfers with low thrust engines,
that we normalize as 
\begin{equation}
\label{eq:5}
\ddot{q} =-\frac{q}{\| q\|^{3}}+u;
\end{equation}
the control is constrained by $\| u\|\leq\varepsilon$, where
$\varepsilon$ is a small parameter. The phase space, or state space, is the one with coordinates $(q,\dot{q})$.
Let $K=\frac12\,\|\dot{q}\| ^{2}-1/\| q\|$ be the mechanical energy of the uncontrolled system and $X$ be the elliptic
domain:
\[
X=\{K<0,q\wedge\dot{q}\neq0\}\,.\]
For the free motion ($u = 0$), the solutions that lie in $X$ are ellipses ---or more precisely closed curves that project
on the $q$ component as ellipses--- and they form a foliation of $X$.

In this domain, we may chose coordinates $(x,l)$ where $x$ is made of independent first integrals
of the uncontrolled motion (so that $x$ describes the geometry of the ellipses) and the ``longitude'' $l$ defines the
position of the spacecraft on this ellipse; $(q,\dot q)$ can be expressed in terms of $(x,l)$ and vice versa.
Restricting to the coplanar case, where $q$ and $\dot q$ have dimension 2 and $x$ has dimension 3, the system can be written as
\begin{displaymath}
\dot{x}=\sum_{i=1,2}u_{i}F_{i}(x,l)\,,\ \ \ \ 
\dot{l} =\Omega(x,l)\,,
\end{displaymath}
where the control $u=(u_{1},u_{2})$ is the coordinates of the original acceleration $u$ in some frame
$F_{1},F_{2}$, e.g., the tangential/normal frame (the vector fields $F_{1},F_{2}$ are another basis of the distribution
spanned by $\partial/\partial \dot q_1, \partial/\partial \dot q_2$ in the original cartesian coordinates).
In these coordinates, the free motion is $\dot x=0,\dot{l} =\Omega(x,l)$; there may be a control term in $\dot l$ too but
we neglect it for clarity.

The energy minimization problem is the one of minimizing a quadratic criterion $\int \|u\|^2\mathrm{d}t$ for fixed
initial and final value of $x$, and free $l$; it was analyzed from the averaging point of view in a series of articles
\cite{Edel64,Edel65}, \cite{Geff-Epe97,Geff97th}, \cite{Bonn-Cai09forum}.
The Pontryagin maximum principle yields (for any type of cost: energy, final time or others) an Hamiltonian on the cotangent bundle of the state space with the property
that a minimizing trajectory must be the projection of an integral curve of the Hamiltonian vector field.
For energy minimization, this Hamiltonian is
\[
H(x,p,l)={\textstyle\frac12}(H_{1}(x,p,l)^{2}+H_{2}(x,p,l)^{2})\]
where $H_{i}(x,p,l)=\langle \,p,F_{i}(x,l)\rangle$ are the Hamiltonian lifts of
the vector fields $F_i$ and $p$ is the vector of costate variables of the same dimension as the state vector.

As the bound $\varepsilon$ tends
to zero, the time needed to reach a given orbit tends to infinity.
During this very long time, the variable $x$ move slowly because the control is small while variables like
$l$ move fast thanks to the term $\Omega$; this yields ill conditioned integration if numeric methods are used.
It may be shown that there is an average Hamiltonian
\[
H(x,p)=\frac{1}{2\pi}\int_{0}^{2\pi}\varpi(x,l)H(x,p,l)\mathrm{d}l\,,\]
with $\varpi$ some weight function to be determined, that eliminates the fast variable $l$ and whose Hamiltonian flow
gives a remarkably good approximation of the movement of $x$ in the original system if $\varepsilon$ is indeed small.
It sometimes leads to explicit formulas, and is anyway much better conditioned numerically because the fast
variable has been eliminated.

We shall recall briefly these facts but are more interested in studying qualitatively this new Hamiltonian.
We refer the reader to \cite[\S52]{Arno89} (although no control is considered there) for details on this approximation and its validity.
It turns out that it is quadratic definite positive with respect to $p$ and hence derives from a Riemannian metric on
$X$; furthermore, the coefficients of this metric can be
explicitly computed. In the coplanar case the geodesic flow is Liouville
integrable and the metric associated to a subproblem related to transfer from an arbitrary orbit (in $X$) to a circular
one is even flat: in suitable coordinates the minimizing
solutions are straight lines
\cite{Bonn-Cai09forum}. Moreover this result is still true if the
thrust is oriented only in the tangential direction \cite{Bonn-Cai-Duj06b}.

The same averaging technique can be applied in the minimum time case. The non averaged Hamiltonian reads
$\sqrt{H_{1}^{2}(x,p,l) + H_{2}(x,p,l)^{2}}$ and again an averaged Hamiltonian may be constructed:
\[
H(x,p)=\frac{1}{2\pi}\int_{0}^{2\pi}\varpi(x,l)\sqrt{H_{1}^{2}(x,p,l) + H_{2}(x,p,l)^{2}}\,\mathrm{d}l\,.\]
Like in the energy case, this Hamiltonian derives from a metric on $X$, i.e. the data of a norm on each tangent space to $X$;
however, unlike in the energy case and as observed in the article \cite{Bomb-Pom13}, these norms are not associated with inner products on
these tangent spaces ---this defines a Finsler metric \cite{Bao-Che-She00}, not necessarily Riemannian--- and are not everywhere smooth.
Technical problems involved in going from Riemannian to non
smooth Finsler geometry make the computations of time minimal transfer towards circular
orbits a complicated problem. 

The objective of this article is to make a preliminary qualitative
description of the time minimum transfers and to compare them with the energy minimum ones:
section~\ref{sec:prelim} recalls the equations and the computation of the
average Hamiltonians; section~\ref{sec:energy} recalls the results from
\cite{Bonn-Cai09forum,Bonn-Cai-Duj06b} on the minimum energy problem; section~\ref{sec:Tmin} provides a new analysis of the minimum time problem, for transfers to
circular orbits, and in particular proves that the elliptic domain is geodesically convex in this case;
section~\ref{sec:compa} explains why that proof fails in the minimum energy problem, which is consistent with the
non-convexity mentioned in \cite{Bonn-Cai09forum}.

\section{Preliminaries}
\label{sec:prelim}

\subsection{Hamiltonian formalism, Pontryagin maximum principle}

The goal of this paper is to study some Hamiltonian systems associated to optimal control problems.
For the sake of self containedness, let us sketch the relation to the optimal control problems.

Consider the smooth control system $\dot{x} = f(x,u,t)$ for $x \in X$, an $n$-dimensional manifold, $t \in \mathbb{R}$ and $u \in B\subset\mathbb{R}^{\ell}$.\\

An \emph{optimal control problem} on $X$ associated with the control system 
$\dot{x} = f(x,u,t)$  is, for instance, the problem of finding relative to the given points $x_0,x_T$
the trajectory $x(\cdot)$ and control $u(\cdot)$, and possibly the final time $T$ if it is not specified, such that
\begin{equation}\label{CntrlPrbm}
\begin{tabular}{ r l}
\(\dot{x}\)&\( = \;\,\, f(x,u,t),\qquad x \in X,\; (u_1,u_2,...,u_{\ell}) \in B\subset\mathbb{R}^{\ell}\)\\
\(x(0)\)&\( =\;\; x_0,\quad x(T) = x_T\)\\
\(\mathcal{J}\)&\( =\;\; \int_{0}^{T}\mathcal{L}(x(t),u(t))dt\rightarrow \mbox{Min}.\)
\end{tabular}
\end{equation}
We call ``minimum time'' the problem where $\mathcal{L}(x,u)=1$ and $T$ is free, and ``minimum energy'' the one
where $T$ is fixed and $\mathcal{L}(x,u)=\|u\|^2$.

The Hamiltonian of the optimal control problem (\ref{CntrlPrbm}) is the function
\[\mathcal{H}(x,p,u,p_0, t)= p_0\mathcal{L}(x,u) + \langle\, p,f(x,u,t)\rangle\,\]
where $p$ is a vector of costate variables  (the adjoint vector) of the same dimension as the state variables $x(t)$, and $p_0$ 
is either $0$ or $-1$.
The Pontryagin maximum principle \cite{Pont-Bol-Gam-M62} (see also \cite[Chap. 6]{Bonn-Fau-Tre06} for applications to
the problems we consider here) is a powerful necessary condition for optimality, that states
the following: if $(x(\cdot),u(\cdot))$ is an optimal trajectory-control pair of the above optimal control problem on a
time interval $[0,T]$, then it can be lifted to a parameterized curve $t\mapsto(x(t),p(t))$ on the cotangent bundle
$T^\star X$ ($p$ is the adjoint vector, or the vector of costate variables) that satisfies, for almost all time and
either for $p_0=0$ or for $p_0=-1$,
\begin{align}\nonumber
  &\dot x(t)=\frac{\partial \mathcal{H}}{\partial p}(x(t),p(t),u(t),p_0, t)=f(x(t),u(t),t)
\\\label{eq:48}
  &\dot{p}(t)=-\frac{\partial \mathcal{H}}{\partial p}(x(t),p(t),u(t),p_0, t)
\end{align}
and, for almost all $t$, $\mathcal{H}(x(t),p(t),u(t),p_0, t)$ is the maximum of $\mathcal{H}(x(t), p(t), u,$ $p_0, t)$
with respect to $u\in B$.
The solutions where $p_0=0$ are called abnormal. Let us assume $p_0=-1$.

In the problems we consider here, we are in the nice situation where for all $(x,p,t)$, or almost all $(x,p,t)$, there is
a unique $u^\star(x,p,t)$ such that 
$$
H(x,p,t)=\mathcal{H}(x,p,u^\star(x,p,t),-1, t)=\max_{u\in B}\mathcal{H}(x,p,u,-1, t)
$$
(the second equality is a property of $u^\star(x,p,t)$; the first equality is the definition of $H$ from $\mathcal{H}$
and $u^\star$).
In that case, one may sum up the above in the following way: if $(x(\cdot),u(\cdot))$ is an optimal trajectory, then
$x(.)$ may be lifted to a solution of the Hamiltonian vector field associated to $H$ on $T^\star X$:
\begin{equation}
  \label{eq:49}
  \dot{x} = \frac {\partial H}{\partial p}(x,p,t), \quad \dot{p} =- \frac {\partial H}{\partial x}(x,p,t).
\end{equation}
The situation is even nicer if $u^\star$ is a smooth function of $x,p,t$; if not, one must be careful about existence
and uniqueness of solutions of solutions to this differential equation.

We kept the above time-varying system because we will encounter time-periodic Hamiltonians that we average with respect to time,
or with respect to a variable that we may view as a new time.

\subsection{Coordinates}

First of all, we recall the equations describing the planar controlled Kepler
problem in the elliptic case (mechanical energy $K$ is negative).

If we chose as coordinates $(n,e_x,e_y,l)$ where $n$ is the mean movement ($n=\sqrt{1/a^{3}}=(-2K)^{3/2}$; $a$ is the semi-major axis), $(e_x,e_y)$
are the coordinates of the eccentricity vector in a fixed frame and $l$ is the ``longitude'', or the polar angle with
respect to a fixed direction, then the elliptic domain is given by $\{n>0,\,{e_x}^2+{e_y}^2<1\}$. The control system is
described by the Gauss equations, where $u_t,u_n$ are the coordinates of the control in the tangential-normal frame:
\begin{subequations}
\label{eq:equinox}
  \begin{align}
    \label{eq:6}
\displaybreak[0]
    &\dot{n}=-3n^{2/3}\frac{\sqrt{1+2\,(e_x\cos l+e_y\sin l)
        +{e_x}^{2}+{e_y}^{2}}}{\sqrt{1-{e_x}^{2}-{e_y}^{2}}}\,u_{t} \\
    \nonumber &\dot{e}_x= n^{-1/3}\frac{\sqrt{1-{e_x}^{2}-{e_y}^{2}}}{\sqrt{1+2\,(e_x\cos l+e_y\sin l)
        +{e_x}^{2}+{e_y}^{2}}} 
\\
\displaybreak[0]
&\label{eq:7}\hspace{7em}\times\!\!\left[2\,(\cos l+e_x)\,u_t -\frac{\sin l+2\,e_y+2 e_x
        e_y\cos l-({e_x}^{2}-{e_y}^{2})\sin l}{\sqrt{1-{e_x}^{2}-{e_y}^{2}}}\,u_n\right]
    \\
    \nonumber &\dot{e}_y= n^{-1/3}\frac{\sqrt{1-{e_x}^{2}-{e_y}^{2}}}{\sqrt{1+2\,(e_x\cos l+e_y\sin l)
        +{e_x}^{2}+{e_y}^{2}}} 
\\
\displaybreak[0]
&\label{eq:8}\hspace{7em}\times\!\!\left[2\,(\sin l+e_y)\,u_t
      -\frac{\cos l+2\,e_x+({e_x}^{2}-{e_y}^{2}) \cos l+2 e_x e_y\sin l}{\sqrt{1-{e_x}^{2}-{e_y}^{2}}}\,u_n\right] \\
    \label{eq:9}
    &\dot{l}=n\,\frac{(1+e_x\cos l+e_y\sin l)^{2}}{(1-e^{2})^{3/2}}.
  \end{align}
\end{subequations}

Instead of $e_x,e_y$, it will be more convenient to use the eccentricity $e$ and the argument
of the pericenter $\omega$ (not defined if $e=0$), defined by
\begin{equation}
  \label{eq:10}
  e_{x}=e\cos \omega,\ e_{y}=e\sin \omega\;.
\end{equation}
The equations become:
\begin{subequations}
  \label{eq:42}
  \begin{align}
    \label{eq:1}
    &\dot{n}=-\frac{3n^{2/3}}{\sqrt{1-e^{2}}}\left[\sqrt{1+2e\cos v+e^{2}}\,u_{t}\right]
    \\
    \label{eq:2}
    &\dot{e}=\frac{\sqrt{1-e^{2}}}{\sqrt[3]{n}}\frac{1}{\sqrt{1+2e\cos v+e^{2}}}\left[2(e+\cos v)\,u_{t}-\sin
      v\frac{1-e^{2}}{1+e\cos v}\,u_{n}\right]
    \\
    \label{eq:3}
    &\dot{\omega}=\frac{\sqrt{1-e^{2}}}{e\sqrt[3]{n}}\frac{1}{\sqrt{1+2e\cos v+e^{2}}}\left[2\sin v \,u_{t}
          +\frac{2e+\cos v+e^{2}\cos v}{1+e\cos v}\,u_{n}\right] 
    \\
    \label{eq:4}
    &\dot{l}=n\,\frac{(1+e\cos v)^{2}}{(1-e^{2})^{3/2}}.
  \end{align}
\end{subequations}
The angle $v$ is the true anomaly
\begin{equation}
  \label{eq:24}
  v=l-\omega.
\end{equation}
In these coordinates, the elliptic domain is
\begin{equation}
  \label{eq:47}
X=\{x=(n,e,\omega)\,,\  n>0,\,0\leq e<1,\,\omega\in S^{1}\}\;.  
\end{equation}

\begin{remark}[Transfer towards a circular orbit]
  \label{rmk:omega}
In the transfer ``towards a circular orbit'' (or merely if we do not take into account the direction of the semi-major
axis during the transfer), we may use these coordinates although they are singular at $e=0$, because the variable
$\omega$ may simply be ignored; this is possible because it is a cyclic variable, i.e. it does not influence the
evolution of the other variables $(n,e,v)$.
In the variables $(n,e)$, the elliptic domain is:
\begin{equation}
  \label{eq:33}
  \mathcal{X}=\{(n,e),\; 0<n<+\infty,\;-1<e<1\}\;.
\end{equation}
The fact that negative values of $e$ are allowed comes from identifying $(-e,\omega)$ with $(e,\omega+\pi)$, or, equivalently, considering that
$(e_x,e_y)$ (see \eqref{eq:10}) lies on a line of fixed arbitrary direction instead of a half-line.
This line may for instance be $\{e_y=0\}$, and $\mathcal{X}$ is then identified with 
$\{(n,e_x,e_y),\,n>0,-1< e <1, e=e_{x},e_{y}=0\}$.
\end{remark}

Equations \eqref{eq:1}-\eqref{eq:4} read:
\begin{equation}
  \label{eq:11}
\dot{x} =\sum_{1\leq i\leq 2} u_i\,F_i(x,l),\;\;\dot{l}=\Omega(x,l)
\end{equation}
where $u_1, u_2$ stand for $u_n,u_t$,  $x = (n,e,\omega)$, the vectors $F_1,F_2$ are readily obtained from \eqref{eq:1}-\eqref{eq:3}, and
\begin{equation}
  \label{eq:16}
  \Omega(x,l)=n\,\frac{(1+e\cos(l-\omega))^{2}}{(1-e^{2})^{3/2}}\ .
\end{equation}

One way to introduce averaging is to use the so-called ``mean eccentric anomaly''. The eccentric anomaly is $\EE$,
related to $e$ and $v$ by
\begin{equation}
\label{eq:20}
  \tan\frac v 2\ =\ \sqrt{\frac{1+e}{1-e}}\tan\frac E 2\,
\end{equation}
and the mean eccentric anomaly is $\EE-e\sin \EE$; the Kepler equation (third Kepler law) implies that, when the control is zero, 
\[
\EE-e\sin \EE=n\,t\,,\]
$t=0$ being the time at the pericenter.
Introducing (see for instance \cite[sec. 3.6.3]{Bonn-Fau-Tre06})
$$x_{0}=(\EE-e\sin \EE)/n\,,$$
one has $\dot{x}_0 =1$ if $u=0$, i.e. the variable $x_0$ behaves like time modulo an additive constant;
this is an implementation of the flow-box theorem. In the coordinates $(x,x_0)$, the system becomes
\[\dot{x}=\sum_{i=1,2}u_{i}\widehat{F}_{i}(x,x_{0}),\;\;\;\,\dot{x}_0 =1+\sum_{i=1,2}u_iG_{i}(x,x_{0}).\]
 Due to the implicit relation between $\EE$ and $x_0$, the practical derivation of such equations is complicated, but
 they will be useful in formally identifying averaging with
respect to $l\in[0,2\pi]$ and averaging with respect to $t\in[0,2\pi/n]$.

We define the Hamiltonian lifts ($i=1,2$):
\begin{equation}
  \label{eq:12}
  H_{i}(x,p,l)=\langle p,F_{i}(x,p,l)\rangle\,,\ \ \ \widehat{H}_{i}(x,p,x_{0})=\langle p,\widehat{F}_{i}(x,p,x_{0})\rangle\,.
\end{equation}

\subsection{Averaging}

Using the previous equations and rescaling the control with $u=\varepsilon v$
to introduce the small parameter, the trajectories parameterized by
$x_{0}$ are solutions of 

\[
\frac{dx}{dx_{0}}=\frac{\varepsilon\sum_{i=1,2}v_{i}\widehat{F}_{i}(x,x_{0})}{1+\varepsilon\sum_{i=1,2}v_{i}G_{i}(x,x_{0})},\]

which is approximated for small $\varepsilon$ by 

\[
\frac{dx}{dx_{0}}=\varepsilon\sum_{i=1,2}v_{i}\widehat{F}_{i}(x,x_{0}).\]
For this system, we consider the following minimization problems:
\begin{equation*}
\begin{tabular}{ c l }
\(\bullet\) & \(\displaystyle \mbox{Energy}\;:\; \min_{v}\varepsilon^{2}\int_{0}^{x_{0}}\sum_{i=1,2}v_{i}^{2}\mathrm{d}t\)\\[0.3cm]
\(\bullet\) & \(\displaystyle \mbox{Time}\; :\; \min_{v}\;x_{0}, \| v\|\leq 1. \)\\
\end{tabular}
\end{equation*}

Applying the Pontryagin maximum principle leads to the following respective
Hamiltonians (normal case in the energy minimization problem), 
\begin{equation}
  \label{eq:44}
  \begin{split}
    H_{\mathrm{e}}(x,p,x_{0})=\sum_{i=1,2}\widehat{H}_{i}(x,p,x_{0})^{2}
\,,\ \ \ \ 
    H_{\mathrm{t}}(x,p,x_{0})=\sqrt{\sum_{i=1,2}\widehat{H}_{i}(x,p,x_{0})^{2}}\,,
  \end{split}
\end{equation}
where the lifts $\widehat{H}_{i}$, defined by \eqref{eq:12}, are periodic
with respect to $x_{0}$ with period $2\pi/n$. 
\begin{remark}[Tangential thrust]
  \label{rmk:tangential}
If the normal component $u_n$ is forced to be zero, there is a single term in the sums in 
\eqref{eq:44}, and these equations become
\quad
$H_{\mathrm{e}}={\widehat{H}_{1}}^{2}$,
\quad
$H_{\mathrm{t}}=\left|\widehat{H}_{1}\right|$. The considerations in the present section are valid both in the full
control case and in the ``tangential thrust'' case.
\end{remark}
The respective averaged Hamiltonians are
\begin{gather}
  \label{eq:13}
  H_{\mathrm{e}}(x,p)=\frac{n}{2\pi}\int_{0}^{2\pi/n}H_{\mathrm{e}}(x,p,x_{0})dx_{0} \\
  \label{eq:14}
  H_{\mathrm{t}}(x,p)=\frac{n}{2\pi}\int_{0}^{2\pi/n}H_{\mathrm{t}}(x,p,x_{0})dx_{0}\,.
\end{gather}
(for ease of notation we use $H_\mathrm{e}, H_\mathrm{t}$ to represent both the Hamiltonians and the averaged Hamiltonians, although the inputs into these functions are different).
These may be re-computed in terms of $H_1,H_2$. Unlike when $\widehat{H}_1,\widehat{H}_2$ are used in the computation, using the Hamiltonian lifts $H_1, H_2$ allows for an explicit expression of
the averaged Hamiltonians $H_\mathrm{e}(x,p), H_\mathrm{t}(x,p)$. Making the change of
variables $x_0=\Xi(e,\omega,l)$ ---with $\Xi$ deduced from $x_{0}=(\EE-e\sin \EE)/n$, \eqref{eq:20} and
\eqref{eq:24}--- in the integral, and using the facts that $\partial\Xi/\partial l=1/\Omega(x,l)$ and 
$$
\widehat{H}_\mathrm{e}(x,p,\Xi(e,\omega,l))=H_\mathrm{e}(x,p,l)\,,\;\;
\widehat{H}_\mathrm{t}(x,p,\Xi(e,\omega,l))=H_\mathrm{t}(x,p,l)\,,
$$
then, using \eqref{eq:16},
\begin{gather}
  \label{eq:25}
  H_{\mathrm{e}}(x,p)=\frac{(1-e^{2})^{3/2}}{2\pi}\int_{0}^{2\pi}\left(\sum_{i=1,2}H_{i}(x,p,l)^{2}\right) \frac{\mathrm{d}l}{(1+e\cos(l-\omega))^{2}} \\
  \label{eq:15}
  H_{\mathrm{t}}(x,p)=\frac{(1-e^{2})^{3/2}}{2\pi}
\int_0^{2\pi}\sqrt{\sum_i H_i(x,p,l)^2}\frac{\mathrm{d}l}{(1+e\cos(l-\omega))^{2}}\,.
\end{gather}


\begin{remark}
  \label{rem-smallthrust} In the original system, the control is ``small'' (parameter $\varepsilon$).
The average system that we study in the next sections can be seen as a limit as $\varepsilon\to0$.

The smaller $\varepsilon$ is, the better the average system approximates the real system, but neither the results of this
paper not any analysis or simulation in the next sections depend on the size of $\varepsilon$, that is on the magnitude
of the thrust.
\end{remark}

\paragraph{Singularities.}
Let us explain how the non smoothness is a result of the averaging of singularities of a control
system.
Consider the time minimal control problem for a generic smooth system of the
form 
\[\dot{x} =F_{0}(x)+\sum_{i=1,m}u_{i}F_{i}(x),\;\;\| u\|\leqq1.\]
 Moreover assume for simplicity that the control distribution $D=\Span\{F_{1}, \ldots$, $F_{m}\}$
is involutive. From the maximum principle in this case, the
extremal control is defined by $u_{i}=\frac{H_{i}(x,p)}{\sqrt{\sum_i H_{i}^{2}(x,p)}}$
where $H_{i}(x,p)$ are the Hamiltonian lifts of $F_{i}(x)$. More
complicated extremals are related to the switching surface $\Sigma:$
$H_{i}=0.$ Observe that in the single-input case the control is given
by $u_{1}=\sign H_{1}(x,p)$ and meeting the surface $\Sigma$ transversally
corresponds to a regular switching. This can be generalized to the
multi-input case. More complicated singularities can occur in the
non transversal case, for instance in relation with singular trajectories
of the system (contained by definition in the surface $\Sigma$) \cite{Bonn-Sug12book}.

\section{The analysis of the averaged systems for minimum energy}
\label{sec:energy}

First of all we recall the results from the energy case \cite{Bonn-Cai09forum} . The energy minimization problem is expressed as
\[
\int_{0}^{l_{f}}\left(u_{1}^{2}(t)+u_{2}^{2}(t)\right)dt\rightarrow \mbox{Min},\]
 where we fix the final cumulated longitude $l_{f}$ (this is slightly different from fixing the transfer time).

\subsection{The coplanar energy case}

In this case the averaged system can be computed explicitly by quadrature,
and we have the following proposition.

\begin{proposition}
In the coordinates $(n,e,\omega)$ the averaged Hamiltonian (up
to a positive scalar) is given by
\begin{equation}
  \label{eq:30}
  H_\mathrm{e} =\frac{1}{n^{5/3}}[18n^{2}p_{n}^{2}+5(1-e^{2})p_{e}^{2}+\frac{5-4e^{2}}{e^{2}}p_{\omega}^{2}]
\end{equation}
 where the singularity $e=0$ corresponds to circular orbits. In particular
$(n,e,\omega)$ are orthogonal coordinates for the Riemannian metric
associated to $H$, namely

\[
g=\frac{1}{9n^{1/3}}\mathrm{d}n^{2} + \frac{2n^{5/3}}{5(1-e^{2})}\mathrm{d}e^{2} + \frac{2n^{5/3}}{5-4e^{2}}\mathrm{d}\omega^{2}.\]
\end{proposition}

 Further normalizations are necessary to capture the main properties
of the averaged orbital transfer.

\begin{proposition}
In the elliptic domain we set 
\[
r=\frac{2}{5}n^{5/6},\varphi=\arcsin e\]
and the metric is isometric to 
\[
g=\mathrm{d}r^{2}+\frac{r^{2}}{c^{2}}(\mathrm{d}\varphi^{2}+G(\varphi)\mathrm{d}\omega^{2})\]
where $c=\sqrt{2/5}$ and $G(\varphi)=\frac{5\sin ^{2}\varphi}{1+4\cos ^{2}\varphi}.$ 
\end{proposition}

\subsection{Transfer towards circular orbits}
As noticed in Remark~\ref{rmk:omega}, for such transfers we may ignore the cyclic variable $\omega$ and allow negative $e$. In this case, the
elliptic domain is the $\mathcal{X}$ given by \eqref{eq:33}. The metric above then reduces to 
\[
g=\mathrm{d}r^{2}+r^{2}d\psi^{2}\,,\ \ \text{with}\ \ \psi=\varphi/c\]
defined on the domain $\{(r,\psi),\,0<r<+\infty,-\frac\pi{2c}<\psi<\frac\pi{2c}\}$; it is
a polar metric isometric to the flat metric $dx^{2}+dz^{2}$
if we set $x=r\sin \psi$ and $z=r\cos \psi$. Flatness in the original coordinates can be checked
 by computing the Gauss curvature. We deduce
the following theorem:
\begin{theorem}
\label{th:droites}
The geodesics of the averaged coplanar transfer towards circular orbits
are straight lines in the domain $\mathcal{X}$ (see \eqref{eq:33}) in suitable coordinates,
namely
\[
x=\frac{2^{3/2}}{5}n^{5/6}\sin (\frac{1}{c}\arcsin e),\;z=\frac{2^{3/2}}{5}n^{5/6}\cos (\frac{1}{c}\arcsin e)\]
with $c=\sqrt{2/5}.$ Since $c<1,$ the domain is not (geodesically)
convex and the metric is not complete.
\end{theorem}

\begin{remark}[Tangential thrust]
  The properties of theorem \ref{th:droites} are still true when the thrust is only in the tangential direction except that the metric has a
  singularity at $e=1$. The formula is 
\[
g=\frac{1}{9n^{1/3}}\;\mathrm{d}n^{2}\;+\;\; \frac{(1+\sqrt{1-e^{2}})n^{5/3}}{4(1-e^{2})}\left[\frac{1}{\sqrt{1-e^{2}}} \;\, \mathrm{d}e^{2}\; + \;\; e^{2}\; \mathrm{d}\omega^{2}\right].\]
We may slightly twist the previous coordinates using $e=\sin\varphi\sqrt{1+\cos^{2}\varphi}$ to get the normal form
$\mathrm{d}r^{2}\;+\;\,(r^{2}/c_{t})(\;\mathrm{d}\varphi^{2}\; +\;  G_{t}(\varphi)\;\mathrm{d}\omega^2),c_{t}=c^{2}=2/5,G_t(\varphi)=\sin^{2}\varphi(\frac{1-(1/2)\sin ^{2}\varphi}{1-\sin
  ^{2}\varphi})^{2}.$
\end{remark}

\section{The analysis of the averaged systems for minimum time}
\label{sec:Tmin}

\subsection{The Hamiltonian}
\label{sec:Tmin:Ham}
We compute $H_\mathrm{t}$ according to \eqref{eq:15}.
The functions $H_i$, $i=1,2$ depend on $n,e,\omega$, $p_n,p_e,p_\omega,l$.
Since we only consider transfer towards a circular orbit, we set $p_\omega=0$ and define $h_1, h_2$ by
\begin{equation}
  \label{eq:18}
  h_i(n,e,p_n,p_e,v)=H_i(n,e,\omega,p_n,p_e,0,\omega+v)\ .
\end{equation}
The right-hand side does not depend on the cyclic variable $\omega$, see Remark~\ref{rmk:omega}.
From here on we will use the subscripts $1$ and $2$ to denote respectively the tangential and normal directions, rather than $t$ and $n$, for ease of notation.
From \eqref{eq:42}, we get
\begin{subequations}
    \label{eq:43}
    \begin{align}
      \label{eq:45}
      &h_1=n^{-1/3}\left(-3n\,p_n\frac{\sqrt{1+2e\cos v+e^{2}}}{\sqrt{1-e^2}}+ 2 p_e 
          \frac{(e+\cos v)\sqrt{1-e^2}}{\sqrt{1+2e\cos v+e^{2}}} \right)
\\
      &h_2=-\,n^{-1/3}\,p_e\, \frac{\sin v \,(1-e^2)^{3/2}}{(1+e\cos v)\sqrt{1+2e\cos v+e^{2}}}
    \end{align}
\end{subequations}
 Note that $\omega$ does not vary
in the integral; the integrand has period $2\pi$ with respect to either $l$ or $v$. This allows us to
make the change of variable $l=\omega+v$ in the integral in \eqref{eq:15}. In the full control case (both tangential and
normal control), the sum in \eqref{eq:15} contains two terms, and we obtain
\begin{equation}
  \label{eq:17}
  H_\mathrm{t}(n,e,p_n,p_e)=\frac{(1-e^{2})^{3/2}}{2\pi}
\!\int_0^{2\pi}
\!\!\!\!
\sqrt{\sum_{i=1}^2  h_i(n,e,p_n,p_e,v)^2}\,\frac{\mathrm{d}v}{(1+e\cos v)^{2}}\,,
\end{equation}
In the tangential thrust case it only
contains $h_1$ ---see remark~\ref{rmk:tangential}--- and we get (the superscript 1 in $H_\mathrm{t}^1$ denotes single input):
\begin{equation}
  \label{eq:17t}
  H_\mathrm{t}^1(n,e,p_n,p_e)=\frac{(1-e^{2})^{3/2}}{2\pi}
\int_0^{2\pi}\bigl|h_1(n,e,p_n,p_e,v)\bigr|\,\frac{\mathrm{d}v}{(1+e\cos v)^{2}}\ .
\end{equation}

In order to highlight some properties of these Hamiltonians, we perform a canonical change of coordinates 
$(n,e,p_n,p_e)\mapsto(\lambda,\varphi,p_\lambda,p_\varphi)$:
$$
n=e^{3\lambda}\,,\ e=\sin\varphi\,,\ p_n=\frac{p_\lambda}{3n}\,,\ p_e=\frac{p_\varphi}{\cos\varphi}
$$
followed by taking $(\rho,\psi)$ as polar coordinated for the adjoint vector $(p_\lambda,p_\varphi)$; we shall never use again
the notations $\lambda$, $p_\lambda$, $p_\varphi $ and directly write the change as
\begin{equation}
  \label{eq:19}
  3\,n\,p_n=\rho\cos\psi\,,\ \ \ \ \sqrt{1-e^2}\,p_e=\rho\sin\psi\,,\ \ \ \ e=\sin\varphi
\,,\ {\textstyle -\frac\pi2<\varphi<\frac\pi2}\,.
\end{equation}
Equations \eqref{eq:17} and \eqref{eq:17t} then yield
\begin{gather}
  \label{eq:21}
  H_\mathrm{t}(n,\sin\varphi,\frac{\rho\cos\psi}{3n},\frac{\rho\sin\psi}{\cos\varphi}) 
    =
    \rho\,n^{-1/3}\,L(\psi,\varphi)
\\
  \label{eq:21t}
  H_\mathrm{t}^1(n,\sin\varphi,\frac{\rho\cos\psi}{3n},\frac{\rho\sin\psi}{\cos\varphi}) 
    =
    \rho\,n^{-1/3}\,M(\psi,\varphi)
\end{gather}
with $L$ and $M$ some functions $\mathcal{C}\to\RR$, where $\mathcal{C}$ is the cylinder
\begin{equation}
  \label{eq:0069}
  \mathcal{C}=\{(\psi,\varphi),\;\psi\in(\RR/2\pi\ZZ),\;\varphi\in\RR,\,-\frac\pi2<\varphi<\frac\pi2\}
=\RR/2\pi\ZZ\times (-\frac\pi2,\frac\pi2)\,.
\end{equation}
The expressions of $L$ and $M$ are, taking  the eccentric anomaly $E$ as the variable of integration instead of $v$ (see \eqref{eq:20}; in particular,
${\mathrm{d}v}/(1+e\cos v)=\mathrm{d}E/\sqrt{1-e^2}$) and restricting the interval of integration from $[0,2\pi]$ to
$[0,\pi]$ because the integrand depends on $\cos E$ only:
\begin{align}
  \label{eq:26}
  &L(\psi,\varphi)
    =
    \frac{1}{\pi}\!\int_0^{\pi}\!\!
    \sqrt{\widetilde{I}(\psi,\varphi,E)} \,\mathrm{d}E\,,
\\
  \label{eq:22}
  &\widetilde{I}(\psi,\varphi,E)= \alpha^{1,1}(\varphi,\cos E)
  \cos^2\!\psi\,
 +2\,\alpha^{1,2}(\varphi,\cos E)\,\cos\psi\sin\psi+\alpha^{2,2}(\varphi,\cos E)\sin^2\!\psi\,,
\\[1ex]
  \nonumber &\ \ \alpha^{1,1}=1-\sin^2\!\varphi\cos^2\!E\,,
\hspace{3em}
\alpha^{1,2}=-2 \cos\varphi\, (1\!-\!\sin\varphi\cos E)\cos E\,,
  \\
  \label{eq:0046}
  &\ \ \alpha^{2,2}=(1\!-\!\sin\varphi\cos E) \left(1-3\sin\varphi \cos E+3\cos^2\!E-\sin\varphi\cos^3 \!E\right)
\end{align}
and
\begin{align}
  \label{eq:26t}
  &M(\psi,\varphi)
    =
    \frac{1}{\pi}\!\int_0^{\pi}\!\!\left|\widetilde{J}(\psi,\varphi,E)\right|\,\mathrm{d}E\,,
\\
\label{eq:22t}
  &\widetilde{J}(\psi,\varphi,E)= \sqrt{\frac{1-\sin\varphi\cos E}{1+\sin\varphi\cos E}}
\,
\Bigl( ( 2\cos\varphi\sin\psi - \sin\varphi\cos\psi )\cos E -\cos\psi\Bigr)\,.
\end{align}

In the sequel we take advantage of the double homogeneity with respect to $\rho$ and $n$ displayed in \eqref{eq:21} and \eqref{eq:21t}.

\subsection{Singularities of the Hamiltonian in the single-input and two-input cases}

According to \eqref{eq:21} and \eqref{eq:21t}, the Hamiltonians $H_\mathrm{t}$ and $H_\mathrm{t}^1$, have the same
degree of smoothness as, respectively the maps $L$ and $M$.

\begin{proposition}
  \label{lem:unic0}
The maps $L:\mathcal{C}\to\RR$
and
$M:\mathcal{C}\to\RR$ are real analytic away from 
\begin{equation}
 \label{eq:38}
 \mathcal{S}=\left\{\!(\psi,\varphi),\,
\tan\psi=\frac{1+\sin\varphi}{2\cos\varphi}\right\}
\cup
\left\{\! (\psi,\varphi)\,,\ \tan\psi=\frac{-1+\sin\varphi}{2\cos\varphi}\right\}.
\end{equation}
They are both continuously differentiable on $\mathcal{C}$, but their differentials are not locally Lipschitz-continuous
on the set $\mathcal{S}$; we have the following moduli of continuity of the differentials: in a neighborhood of a point
$\xi=(\psi,\varphi)\in\mathcal{S}$, in some corrdinates and for a ``small'' $\delta$,
\begin{align}
  \label{eq:46}
  \|\mathrm{d}L(\xi+\delta)-\mathrm{d}L(\xi)\|
&\leq k\,\|\delta\|\,\ln(1/\|\delta\|)
\\
  \label{eq:46t}
  \|\mathrm{d}M(\xi+\delta)-\mathrm{d}M(\xi)\|
&\leq k\,\|\delta\|^{1/2}
\end{align}
\end{proposition}
\begin{proof}
  The set $\mathcal{S}$ is the set of points $(\psi,\varphi)$ such that $\widetilde{I}(\psi,\varphi,E)$ vanishes
  for some value of $E$; hence the integrand in \eqref{eq:26} is real analytic on $\mathcal{C}\setminus \mathcal{S}$ and
  so is $L$. The degree of regularity \eqref{eq:46} for $L$ at points in $\mathcal{S}$ is given in \cite{Bomb-Pom13}.

Let us now treat $M$.
It turns out that $\mathcal{S}$ is \emph{also} the border between the region
  \begin{equation}
    \label{eq:R1}
    \mathcal{R}_1 = \{(\psi,\varphi) \in \mathcal{C}\;:\;\frac{-1 + \sin\varphi}{2\cos\varphi}< \tan\psi < \frac{ 1 +
      \sin\varphi}{2\cos\varphi} \}
  \end{equation}
where the sign of $\widetilde{J}(\psi,\varphi,E)$ does not depend on $E$ and the region
  \begin{equation}
    \label{eq:R2}
    \mathcal{R}_2 = \{(\psi,\varphi) \in \mathcal{C}\;:\;\tan\psi <\frac{-1 + \sin\varphi}{2\cos\varphi} \mbox{ or } \frac{ 1 +
      \sin\varphi}{2\cos\varphi}< \tan\psi \}
  \end{equation}
where $\widetilde{J}(\psi,\varphi,E)$ vanishes for two distinct values of the angle $E$ where it changes sign; these two
values are given by $\cos E=R(\psi,\varphi)$ with
\begin{equation}
  \label{eq:35}
  R(\psi,\varphi)=
\frac {\cos\psi}{P(\psi,\varphi) }\,,\ \ \ 
P(\psi,\varphi)=2\cos\varphi\sin\psi - \sin\varphi\cos\psi 
\end{equation}
(note that \eqref{eq:R1},\eqref{eq:R2} amount to $\mathcal{R}_1 = \{ |R(\psi,\varphi)|>1\}$, $\mathcal{R}_2 = \{
|R(\psi,\varphi)|<1\}$ and $\mathcal{S}$ is the locus where $R=\pm1$). Hence \eqref{eq:26t} yields
\begin{equation}
  \label{eq:36}
  M(\psi,\varphi)=
  \begin{cases}
    \frac{-\sign\cos\psi}{\pi}\!\int_0^{\pi} \widetilde{J}(\psi,\varphi,E)\,\mathrm{d}E&\mbox{on }\mathcal{R}_1\,,
\\[1ex]
    \frac{\sign P(\psi,\varphi)}{\pi}\left(
\int_0^{\arccos R(\psi,\varphi)} \widetilde{J}(\psi,\varphi,E)\,\mathrm{d}E
-
\int_{\arccos R(\psi,\varphi)}^\pi \widetilde{J}(\psi,\varphi,E)\,\mathrm{d}E
\right)&\mbox{on }\mathcal{R}_2\,.
  \end{cases}
\end{equation}
It is therefore clear that $M$ is real analytic on $\mathcal{C}\setminus \mathcal{S}=\mathcal{R}_1\cup\mathcal{R}_2$.
The singularity of $M$ on $\mathcal{S}$ is not of the type treated in \cite{Bomb-Pom13}, but it is clear above that the restriction of $M$
to $\mathcal{R}_1$ has a real analytic continuation through $\mathcal{S}$ while its restriction to $\mathcal{R}_2$, on
the contrary, behaves like a square root in a neighborhood of $\mathcal{S}$, whence \eqref{eq:46t}.
\qed
\end{proof}

The properties of the differential of the Hamiltonian are important because it is the right-hand side of the Hamiltonian equation.
Studying these singularities  more precisely is an interesting program that is not yet carried out. 

\subsection{The Hamiltonian flow}
\label{flow-full}
Let us now study the solutions of the Hamiltonian equation associated with the minimum time problem in the
full control or single control (tangential thrust) cases, namely:
\begin{equation}
  \label{eq:23}
  \dot n=\frac{\partial H_\mathrm{t}}{\partial p_n}\,,\ \dot e=\frac{\partial H_\mathrm{t}}{\partial p_e}\,,\ 
  {\dot p}_n=-\frac{\partial H_\mathrm{t}}{\partial n}\,,\ {\dot p}_e=-\frac{\partial H_\mathrm{t}}{\partial e}
\end{equation}
and
\begin{equation}\label{A4}
\dot{n} = \frac{\partial H_\mathrm{t}^1}{\partial p_n},\quad \dot{e} = \frac{\partial H_\mathrm{t}^1}{\partial p_e}, \quad \dot{p}_n = \frac{\partial H_\mathrm{t}^1}{\partial n},
\quad \dot{p}_e = \frac{\partial H_\mathrm{t}^1}{\partial e}.
\end{equation}
with $H_\mathrm{t}$ given by \eqref{eq:17} and $H_\mathrm{t}^1$ by \eqref{eq:17t}.

Specifically, we establish geodesic convexity of the elliptic domain
$\mathcal{X}$ (see \eqref{eq:33}), i.e. any two points in $\mathcal{X}$ can be joined by a extremal curve. This is
contained in the following result:
\begin{theorem}[geodesic convexity]
  \label{th:convex} \label{Th1}
  For any $(n^0,e^0)$ and $(n^1,e^1)$ in $\mathcal{X}$, there exist a time $T\geq0$ and
  a solution $[0,T]\to\mathcal{X}$, $t\mapsto(n(t),e(t),p_n(t),p_e(t))$ of \eqref{eq:23} (resp. of \eqref{A4}) such that $(n(0),e(0))=(n^0,e^0)$
  and $(n(T),e(T))=(n^1,e^1)$.
\end{theorem}

\medskip

In order to ease the proof, let us write \eqref{eq:23} and \eqref{A4} in other coordinates.
\begin{proposition}
  \label{prop-varpsi}
  In the coordinates $(n,\varphi,\psi,\rho)$ defined by \eqref{eq:19}, and after a time re-parametrization
  \begin{equation}
    \label{Time}
    \mathrm{d}t=n^{1/3}\, \mathrm{d}\tau\,,
  \end{equation}
  equation \eqref{eq:23} (resp. equation \eqref{A4}) becomes
  \begin{equation}\label{A5}
\frac{\mathrm{d}\psi}{\mathrm{d}\tau} = a(\psi,\varphi) \,,\ \ \ 
\frac{\mathrm{d}\varphi}{\mathrm{d}\tau} = b(\psi,\varphi) \,,\ \ \ 
\frac{\mathrm{d}n}{\mathrm{d}\tau} =-\,3n\,c(\psi,\varphi) \,,
  \end{equation}  
where $a,b,c$ are given by\footnote{lower indices stand for partial derivatives}:
  \begin{align}
    \nonumber
    &a(\psi,\varphi)=-L(\psi,\varphi)\sin\psi-L_\varphi (\psi,\varphi)\cos\psi\,,
    \\
    \label{eq:0026}
    &b(\psi,\varphi)=\ \ \;L(\psi,\varphi)\sin\psi+L_\psi (\psi,\varphi)\cos\psi\,,
    \\
    \nonumber
    &c(\psi,\varphi)=\ \ \;L(\psi,\varphi)\cos\psi-L_\psi (\psi,\varphi)\sin\psi
  \end{align}
(resp. given by:
  \begin{align}
    \nonumber
    &a(\psi,\varphi)=-M(\psi,\varphi)\sin\psi-M_\varphi (\psi,\varphi)\cos\psi\,,
    \\
    \label{A2}
    &b(\psi,\varphi)=\ \ \;M(\psi,\varphi)\sin\psi+M_\psi (\psi,\varphi)\cos\psi\,,
    \\
    \nonumber
    &c(\psi,\varphi)=\ \ \;M(\psi,\varphi)\cos\psi-M_\psi (\psi,\varphi)\sin\psi \ \ \ \mbox{)}
  \end{align}
and the evolution of $\rho$ is given by:
\begin{align}
  \label{eq:34}
\rho(\tau)=&\,\rho(0)\left(\frac{n(0)}{n(\tau)}\right)^{-1/3}\frac{L(\psi(0),\varphi(0))}{L(\psi(\tau),\varphi(\tau))}
\\\label{39}
\mbox{(resp.}\ \ \ \rho(\tau) =&\, \rho(0)\left(\frac{n(0)}{n(\tau)}\right)^{-1/3}\frac{M(\psi(0),\varphi(0))}{M(\psi(\tau),\varphi(\tau))}
\ \ \ \mbox{).}
\end{align}
The ``time'' $\tau$ is related to the real time $t$ by
\begin{align}
    \label{eq:0054}
    t=&\,
    \frac{n(\tau)^{1/3}\cos\psi(\tau)}{L(\varphi(\tau),\psi(\tau))} -
    \frac{n(0)^{1/3}\cos\psi(0)}{L(\varphi(0),\psi(0))}\,.
\\\label{40}
\mbox{(resp.}\ \ \ t=&\,
\frac{n(\tau)^{1/3}\cos\psi(\tau)}{M(\psi(\tau),\varphi(\tau))} - \frac{n(0)^{1/3}\cos\psi(0)}{M(\psi(0),\varphi(0))}
\ \ \ \mbox{).}
\end{align}
\end{proposition}
\begin{proof}
From \eqref{eq:19} and \eqref{eq:23} (resp. \eqref{eq:19} and \eqref{A4}), one gets
\begin{align}
\nonumber
  &\dot\psi=\frac1\rho\left(3 n \sin\psi \frac{\partial H}{\partial n} - \cos\varphi \cos\psi \frac{\partial H}{\partial
    e}\right)
-\cos\psi\sin\psi\left(\frac1n \frac{\partial H}{\partial p_n}
    +\frac{\sin\varphi}{\cos^2\!\varphi}\frac{\partial H}{\partial p_e}\right),
\\
  \label{eq:27}
&\dot\varphi=\frac1{\cos\varphi}\frac{\partial H}{\partial p_e}\,,\ \ \ \ \ \dot n=\frac{\partial H}{\partial p_n}
\end{align}
where $H$ stands for $H_\mathrm{t}$ (resp. for $H_\mathrm{t}^1$).
Differentiating \eqref{eq:21} (resp. \eqref{eq:21t}) with respect to $n,\varphi,\rho,\psi$ and solving for
$\frac{\partial H_\mathrm{t}}{\partial n}, \frac{\partial H_\mathrm{t}}{\partial e}, \frac{\partial H_\mathrm{t}}{\partial p_n}, \frac{\partial
  H_\mathrm{t}}{\partial p_e} $
(resp. for $\frac{\partial H_\mathrm{t}^1}{\partial n}$, $\frac{\partial H_\mathrm{t}^1}{\partial e}$, $\frac{\partial H_\mathrm{t}^1}{\partial p_n}$, $\frac{\partial
  H_\mathrm{t}^1}{\partial p_e} $), we obtain the latter as linear combinations of $L(\psi,\varphi)$,
$L_\varphi(\psi,\varphi)$, $L_\psi(\psi,\varphi)$ (resp. of $M(\psi,\varphi), M_\varphi(\psi,\varphi), M_\psi(\psi,\varphi)$)
with coefficients depending on $n,\varphi,\rho,\psi$; substituting these expressions into \eqref{eq:27} gives
  \begin{eqnarray*}
    \dot\psi&=& n^{-1/3}\left(-L\sin\psi-L_\varphi\cos\psi\right)\,,
    \\
    \dot\varphi&=&n^{-1/3}\left(L\sin\psi+L_\psi\cos\psi\right)\,,
    \\
    \dot n&=& -3\,n^{2/3}\left(L\cos\psi-L_\psi\sin\psi\right)
\\[1ex]
\mbox{(resp.}\ \ \ \ \ 
    \dot\psi&=& n^{-1/3}\left(-M\sin\psi-M_\varphi\cos\psi\right)\,,
    \\
    \dot\varphi&=&n^{-1/3}\left(M\sin\psi+M_\psi\cos\psi\right)\,,
    \\
    \dot n&=& -3\,n^{2/3}\left(M\cos\psi-M_\psi\sin\psi\right)\ \ \ \mbox{).}
  \end{eqnarray*}
With the new time $\tau$ given by \eqref{Time}, one easily deduces \eqref{A5} and the expressions \eqref{eq:0026}
(resp. \eqref{A2}) of $a,b,c$.
Finally, \eqref{A5} and \eqref{eq:0026} (resp. \eqref{A5} and \eqref{A2}) imply
$
\frac{\mathrm{d}}{\mathrm{d}\tau} \!\left(\! \frac{n^{1/3}\,\cos\psi}{L(\psi,\varphi)} \!\right) =n^{-1/3}
$
(resp. $\frac{d}{d\tau}\left(\frac{n^{1/3}\cos\psi}{M(\psi, \varphi)} \right)= n^{1/3}$), that implies 
\eqref{eq:0054} (resp. \eqref{40}) according to \eqref{Time}.\qed
\end{proof}

The first two equations in \eqref{A5} form an autonomous system of equations
in the two variables $(\psi,\varphi)\in\mathcal{C}$ that will be the core of our analysis; the third one may be
integrated and yields $n(\tau)$:
\begin{equation}
  \label{eq:29}
  n(\tau)=n(0)\exp\Bigl(-3
\int_0^\tau c(\psi(\sigma),\varphi(\sigma))\mathrm{d}\sigma\Bigr)\,.
\end{equation}
The variable $\rho$ (the magnitude of the adjoint vector) plays no role in the
evolution of the other variables, in particular the state $(n,e)$ ($e=\sin\varphi$); this is
a well-known consequence of the Hamiltonian being homogeneous of degree 1 with
respect to the adjoint vector and is anyway obvious from \eqref{A5}.

\bigskip \bigskip

Let us now gather some properties of the maps $a,b,c$, \emph{i.e.} of
the differential equation \eqref{A5}, that are valid both for $a,b,c$ given by \eqref{eq:0026} and for $a,b,c$ given by
\eqref{A2}; they contain all the information to prove Theorem~\ref{Th1}.

\begin{proposition}
  \label{lem-check}The maps $a,b,c$ given by \eqref{eq:0026} satisfy the following properties
  with $\overline\sigma=0$.
The maps $a,b,c$ given by \eqref{A2} satisfy the same properties with $\overline\sigma=\arctan\frac12$.
\begin{enumerate}
\item\label{ass-sym} \textbf{\boldmath Symmetries.} For all $(\psi,\varphi)$ in $\mathcal{C}$,
  \begin{equation}
    \label{eq:0036}
    \begin{array}{l}
      a(\psi+\pi,\varphi)=-a(\psi,\varphi) \,, \\
      b(\psi+\pi,\varphi)=-b(\psi,\varphi) \,, \\
      c(\psi+\pi,\varphi)=-c(\psi,\varphi) \,,
    \end{array}
\hspace{3em}
    \begin{array}{l}
      a(-\psi,-\varphi)=-a(\psi,\varphi) \,, \\
      b(-\psi,-\varphi)=-b(\psi,\varphi) \,, \\
      c(-\psi,-\varphi)=c(\psi,\varphi) \,. 
    \end{array}
  \end{equation}
\item\label{ass-cauchy} \textbf{\boldmath Uniqueness of solutions.} 
  The following differential equation on $\mathcal{C}$:
      \begin{equation}
      \label{eq:0024}
        \dot\psi=a(\psi,\varphi)\,,\ \ \  \dot\varphi=b(\psi,\varphi)
      \end{equation}
  has, for any $(\psi^o\!,\varphi^o)\in\mathcal{C}$, a unique solution $t\mapsto(\psi(t),\varphi(t))$ such that
  $(\psi(0),\varphi(0))=(\psi^o\!,\varphi^o)$, defined on a maximum open interval of definition $(\tau^-\!,\tau^+)$. In this interval,
  $\tau^-<0<\tau^+$ where $\tau^{-}$ is such that either
  $\tau^-=-\infty$ or $\varphi(\tau^-)=\pm\frac\pi2$, and  $\tau^{+}$ is such that either $\tau^+=+\infty$ or 
  $\varphi(\tau^+)=\pm\frac\pi2$. This defines a flow $\Phi$ from an open subset of $\mathcal{C}\times\RR$ to
  $\mathcal{C}$ such that the above unique solution is 
  \begin{equation}
    \label{zeq:12}
    t\mapsto\Phi(\psi^o,\varphi^o,t)\,.
  \end{equation}
\item \label{ass-b} \textbf{\boldmath Sign and zeroes of $b$.}
  There exists a continuous map 
\begin{equation}
  \label{zeq:13}
  Z_b:\,[0,\frac\pi2]\to(-\frac\pi2,0]
\end{equation}
\emph{continuously differentiable on the open interval}  $(0,\frac\pi2)$, such that
\begin{equation}
  \label{zeq:19}
  Z_b(0)=-\overline{\sigma}
\end{equation}
and
\begin{equation}
  \label{zeq:7}
  \left.\begin{array}{r}
    b(\psi,\varphi)=0,\\ \varphi\geq0
  \end{array}\right\}
  \ \Leftrightarrow\ 
  \begin{cases}
    \text{either}&\psi=Z_b(\varphi),\\\text{or}&\psi=\pi+Z_b(\varphi),
\\
    \mbox{or}&\varphi=0\ \ \text{and}\ \ \psi\in[-\overline{\sigma},\overline{\sigma}]\cup[\pi-\overline{\sigma},\pi+\overline{\sigma}]\,.
  \end{cases}
\end{equation}
Furthermore,
\begin{equation}
  \label{zeq:9}
  \left.\begin{array}{r}
    b(\psi,\varphi)>0,\\ \varphi\geq0
  \end{array}\right\}
  \ \Leftrightarrow\ 
  \begin{cases}\text{either}
   & \varphi>0\ \ \text{and}\ \ Z_b(\varphi)<\psi<\pi+Z_b(\varphi)\,, \\\mbox{or}&
   \varphi=0\ \ \text{and}\ \ \overline{\sigma}<\psi<\pi+\overline{\sigma} \,,
  \end{cases}
\end{equation}
\item \label{ass-a} \textbf{\boldmath Sign and zeroes of $a$.}
  One has
\begin{equation}
  \label{zeq:15}
  \begin{array}{ll}
&0<\varphi<\frac\pi2\Rightarrow a(Z_b(\varphi),\varphi)>0\,,
\hphantom{\mbox{and, if }\overline{\sigma}>0}
\\[.6ex]
&a(0,0)=0\,,
\\[1ex]
    \mbox{and, if }\overline{\sigma}>0\,,\ \ \  & -\overline{\sigma}\leq\psi<0\Rightarrow a(\psi,0)>0\,.
  \end{array}
\end{equation}

\item\label{ass-saddle} \textbf{\boldmath Hyperbolic saddle point at $(0,0)$.} The maps $a$ and $b$ are smooth in a neighborhood
  of $(0,0)$ and  
    \begin{equation}
      \label{zeq:16}
      a(0,0)=b(0,0)=0\,,\ \ \ 
    \frac{\partial a}{\partial\psi}(0,0) \,\frac{\partial b}{\partial\varphi}(0,0)-\frac{\partial a}{\partial\varphi}(0,0) \,\frac{\partial b}{\partial\psi}(0,0)<0\,.
    \end{equation}
\item \label{ass-c} \textbf{\boldmath Values of $c$ at equilibria.} 
\begin{equation}
    \label{zeq:18}
    c(0,0)=1\,,\ \ \ c(\pi,0)=-1\,.
  \end{equation}
\item \label{ass-US} \textbf{\boldmath Stable and unstable manifolds of $(0,0)$.}
   There exists continuous maps
\begin{equation}
  \label{zeq:1}
  S:\,[0,\frac\pi2]\to[-\pi,0]
\,,\ 
  U:\,[0,\frac\pi2]\to[0,\pi]
\,,
\end{equation}
\emph{continuously differentiable on the open interval} $(0,\frac\pi2)$, and a number $\overline{\sigma}$ with
\begin{equation}
  \label{zeq:3}
  U(0)=0, \ \ S(0)=-\overline{\sigma},\ \,\overline{\sigma}\geq0\,,
\end{equation}
such that
the stable and unstable manifolds of $(0,0)$ are described by 
\begin{equation}
  \label{zeq:10}
  \begin{array}{l}
    \mathcal{S}^0=\{(S(\varphi),\varphi),\,0\leq \varphi<\frac\pi2\}\cup\,[-\overline{\sigma},\overline{\sigma}]\!\times\!\{0\}
\,
    \cup \,\{(-S(-\varphi),\varphi),\,-\frac\pi2<\varphi\leq 0\}
\\
    \mathcal{U}^0=\{(U(\varphi),\varphi),\,0\leq \varphi<\frac\pi2\}\cup\{(-U(-\varphi),\varphi),\,-\frac\pi2<\varphi\leq 0\}
  \end{array}
\end{equation}
Furthermore, the zeroes of $b$ are positioned  with respect to the stable and unstable manifolds so that
  the maps $S,U,Z_b$ satisfy:
\begin{equation}
  \label{zeq:2}
  0<\varphi<\frac\pi2\Rightarrow S(\varphi)<Z_b(\varphi)<0<U(\varphi)\,.
\end{equation}
\end{enumerate}
\end{proposition}
\begin{proof}
  See Appendix~\ref{sec-proof-check}.\qed
\end{proof}

The following theorem is almost independent of the rest of the paper: it states that for any $a,b,c$
that satisfy the seven conditions established in Proposition~\ref{lem-check}, the differential equation \eqref{eq:0024}
has some properties (that will lead to geodesic convexity); the conditions are of course much more general than the two
cases considered in Proposition~\ref{lem-check}.
Theorem~\ref{th:convex} will be easily deduced from Theorem~\ref{prop:conv}. 
\begin{theorem}
  \label{prop:conv} 
  If $a,b,c$ satisfy the properties of Proposition~\ref{lem-check}, i.e. \eqref{eq:0036} through \eqref{zeq:18}, then, 
  for any $\varphi^0$ and $\varphi^1$ in the interval $(-\pi/2,\pi/2)$ and any $\bar\lambda\in\RR$,
  there exists $\tfin\geq0$ and a solution $(\psi(.),\varphi(.)):[0,\tfin]\to\mathcal{C}$ of
  \eqref{eq:0024} such that
    \begin{equation}
    \label{eq:0027}
    \varphi(0)=\varphi^0\,,\ \ \ \varphi(\tfin)=\varphi^1\,,\ \ \ \int_0^{\tfin}c(\psi(\tau),\varphi(\tau))\mathrm{d}\tau=\bar\lambda\,.
  \end{equation}
\end{theorem}
\noindent\textit{Proof of Theorem~\ref{prop:conv}.}
  See Appendix~\ref{sec-proof-thm}.\qed
\noindent\textit{Proof of Theorem~\ref{th:convex}.}
Pick $n^0,e^0,n^1,e^1$; according to Proposition~\ref{lem-check}, Theorem~\ref{prop:conv} applies to $a,b,c$ defined
either by \eqref{eq:0026}  or by \eqref{A2}. Take
$$
\varphi^0=\arcsin e^0\,,\ \ \varphi^1=\arcsin e^1\,,\ \ \bar\lambda=-\frac13\ln\frac{n^1}{n^0}
$$
and apply this theorem.
Use \eqref{eq:29} to get $n(\tau)$ and \eqref{eq:34} or \eqref{39} to get $\rho(\tau)$ (with some arbitrary $\rho(0)$,
for instance $\rho(0)=1$) and finally
\eqref{eq:19} to get $e(\tau),p_n(\tau),p_e(\tau)$ from $\psi(\tau),\varphi(\tau),n(\tau),\rho(\tau)$. Apply the time reparametrization ($\tau\leadsto t$) given by
\eqref{eq:0054} or \eqref{40}, $T$ being deduced from $\tfin$ in the same way.
According to Proposition~\ref{prop-varpsi}, 
 the obtained $t\mapsto(n(t),e(t),p_n(t),p_e(t))$ satisfies the conclusions of Theorem~\ref{th:convex}.\qed

\subsection{Simulations}
\begin{figure}[ht!]
\centering
\includegraphics[width= 1\textwidth,height= .48\textwidth]{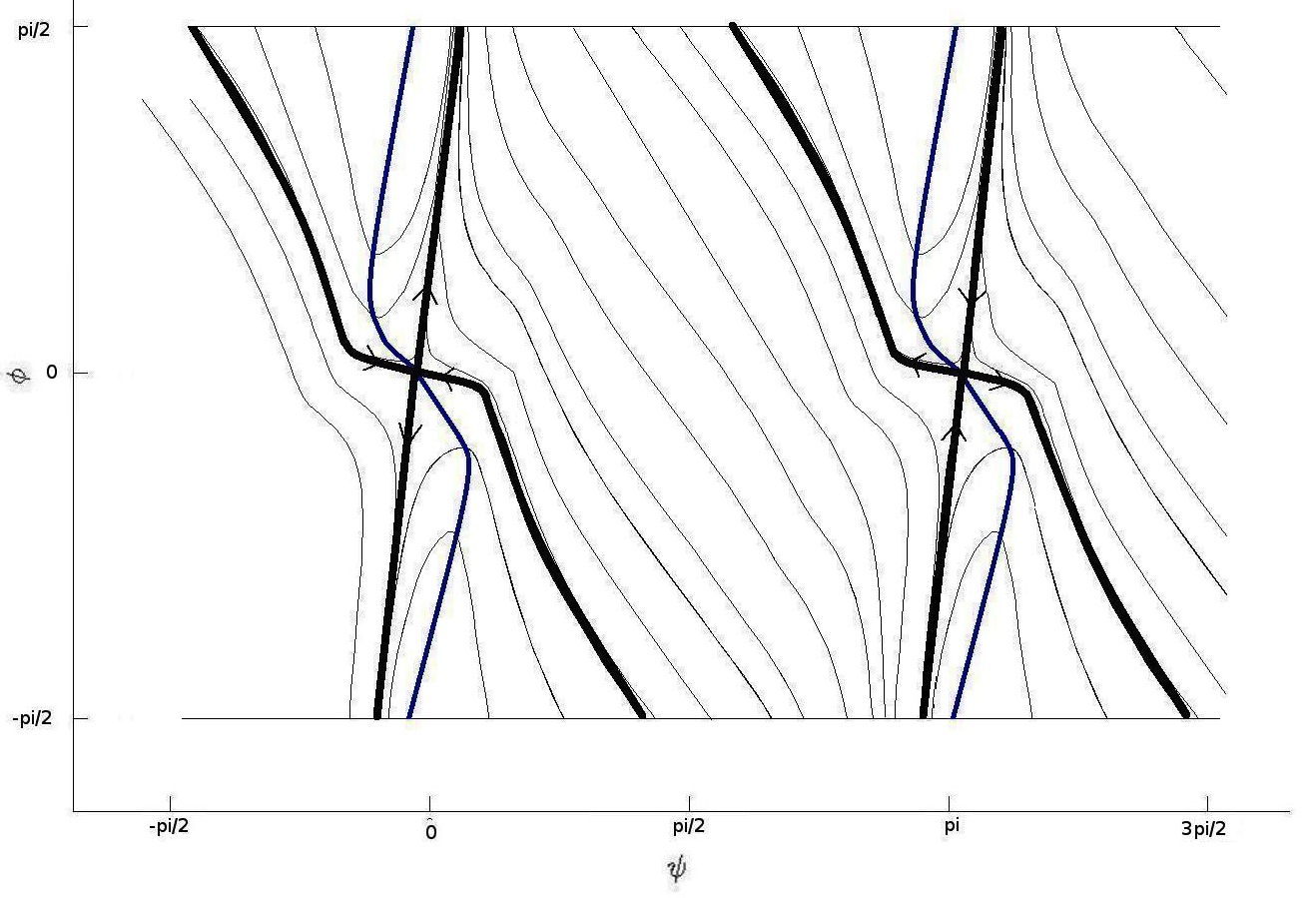}
\caption{Numerical plot (obtained using Matlab) of the stable and unstable manifolds (bold)
 and trajectories through a number of arbitrary initial values in $\mathcal{C}$ for the full control case. The other
 curve shown is $\psi=Z_b(\varphi)$. 
\label{Fig_1}
}
\end{figure}

\begin{figure}[ht!]
\centering
\includegraphics[width=1\textwidth,height= .5\textwidth]{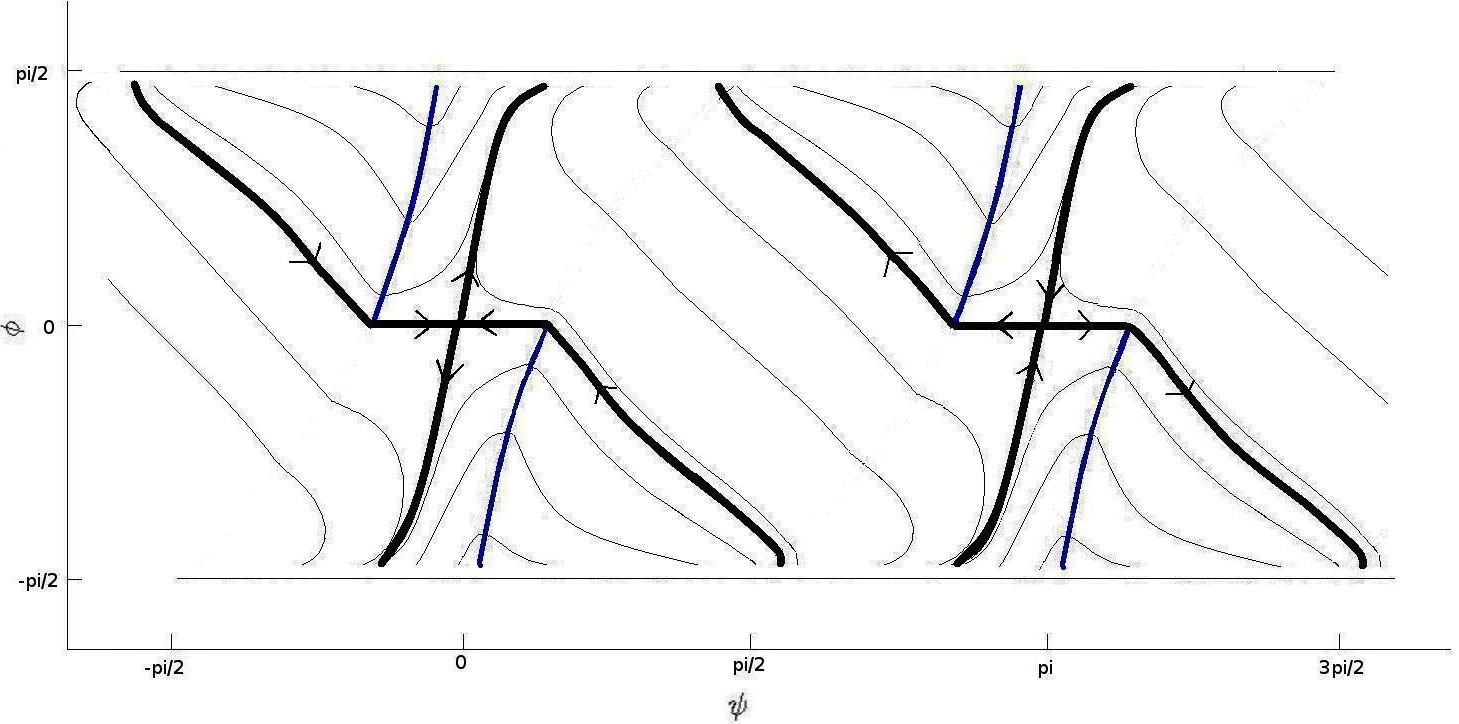}
\caption{Numerical plot (obtained using Matlab) of the stable and unstable manifolds (bold) 
 and trajectories through a number of arbitrary initial values in $\mathcal{C}$ for the tangential case. The other curve shown 
is $\psi=Z_b(\varphi)$. 
\label{Fig_2}
}
\end{figure}

A numerical simulation of the phase portrait of the differential equation \eqref{eq:0024} (or the first two equations
in \eqref{A5}) is displayed in Figure~\ref{Fig_1} in the ``full control case'' where $a$ and $b$ are given by
\eqref{eq:0026}  and in Figure~\ref{Fig_2} in the ``tangential thrust case'' where $a$ and $b$ are given by \eqref{A2}.
This is supposed to be a phase portrait on the cylinder $\mathcal{C}$ (for instance, identify $\{\psi=\frac\pi2\}$ with $\{\psi=\frac{3\pi}{2}\}$).

The thick trajectories are the stable and unstable manifolds of $(0,0)$ and $(\pi,0)$; the other thick curve is the set
of zeroes of $b(\psi,\varphi)$ (i.e. the isocline $\{\dot\varphi=0\}$).
One may check visually the properties established in Proposition~\ref{lem-check}; in particular the unstable manifold of $(0,0)$
is, in both cases, a graph $\varphi\mapsto\psi$ while the stable manifold is also such a graph in the full control case (Figure~\ref{Fig_1})
but not in the tangential thrust case (Figure~\ref{Fig_1}) where it comprises a segment of the $\psi$-axis.

 It can be seen that in both
cases, the cylinder $\mathcal{C}$ is divided into six
regions by these invariant manifolds: one region (called $F$ in Appendix~\ref{sec-proof-thm}) where  all trajectories go
``up'' ($\varphi$ is monotone increasing), one (called $F^+$ in Appendix~\ref{sec-proof-thm}) where  all
trajectories go ``down'', and four other regions (called $E$, $E^\sharp$, $E^+$ and $E^{+\sharp}$ in
Appendix~\ref{sec-proof-thm}) where all trajectories cross once the isocline $\{\dot\varphi=0\}$ so that they go up and
then down or down and then up.

This is exploited in the proof of Theorem~\ref{prop:conv}.
The generic figure \ref{fig:regions} is a drawing used to support that proof, that figures in an illustrative manner the
features contained in the assumptions of Theorem~\ref{prop:conv}, and that can also be observed in the numerical
simulations of the two cases that we are really interested in (Theorem~\ref{Th1}).

\section{Comparison between the minimum-energy and minimum-time cases from the convexity point of view}
\label{sec:compa}

In section \ref{sec:energy} we recalled  some results from \cite{Bonn-Cai09forum} (and previous work by the same
authors); in particular, Theorem~\ref{th:droites} states that the elliptic domain is
\emph{not} geodesically convex for the energy minimization problem, i.e. some pairs of points in $\mathcal{E}$ cannot be
joined by a geodesic.
In that case, in suitable coordinates ($(n^{5/6},\sqrt{5/2}\varphi)$ as polar coordinates), geodesics are straight lines
hence geodesic convexity reduces to usual (affine) convexity, thus the 
simplest way to see this non convexity is to determine the shape on the elliptic domain in these polar coordinates.

\begin{figure}[ht]
  \centering
  \includegraphics[width=0.8\textwidth]{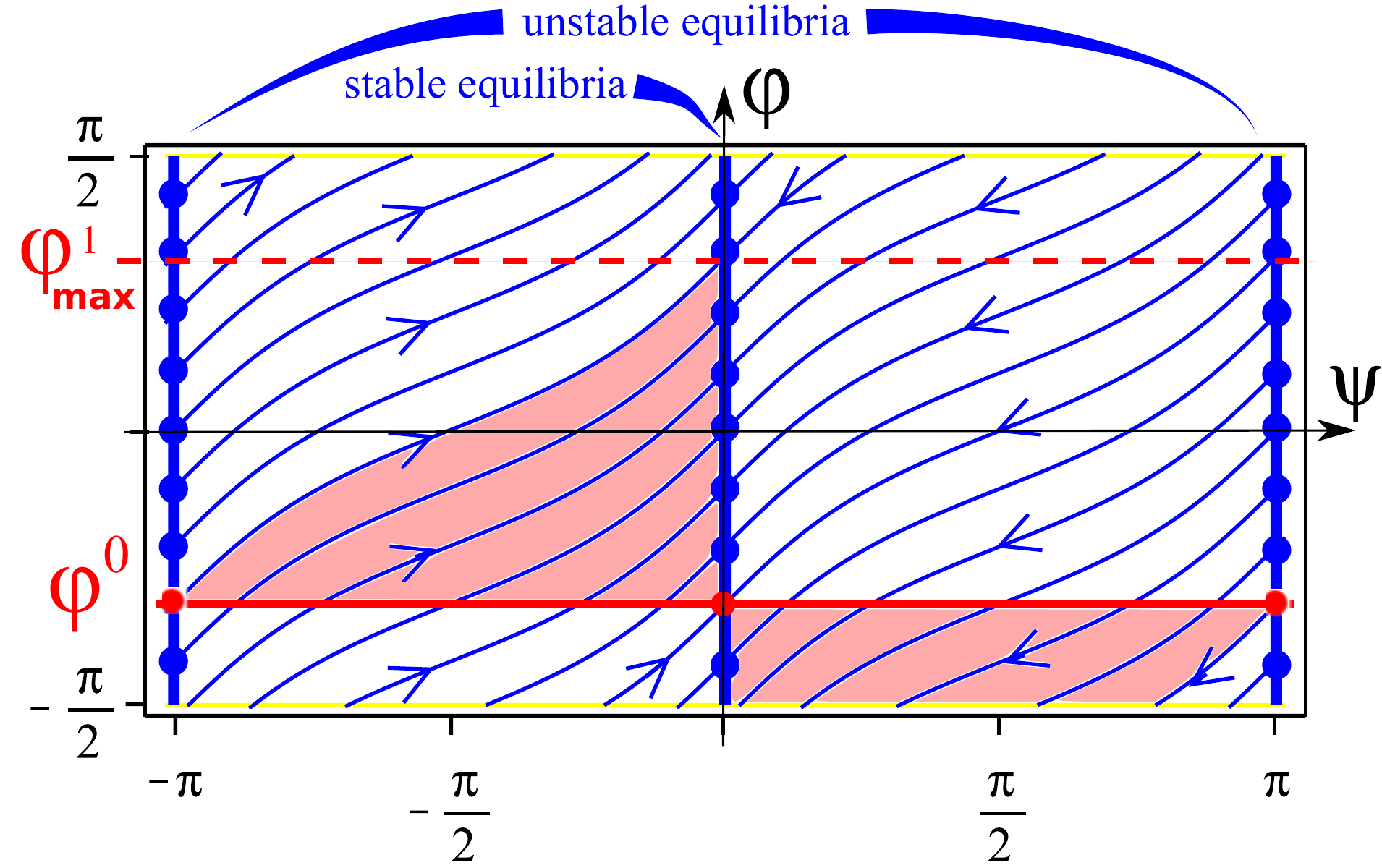}
    \caption{The phase portrait, for energy minimization, in the same coordinates as Figure \ref{Fig_1} and \ref{Fig_2}. There are
    two lines of non isolated equilibria. The darker zone is all the points that can be reached in positive time from
    the line $\{\varphi=\varphi^0\}$, with $\varphi^0$ rather close to $-\frac\pi2$. The highest possible final
    value is $\varphi$ is $\varphi^1_{\mathsf{max}}=\varphi^0+\sqrt{2/5}\,\pi$.}
    \label{fig:L2}
\end{figure}

Here we try to explain why convexity holds in the minimum-time case and not in the minimum-energy case.
Using the coordinates from Theorem~\ref{th:droites} for the time-minimizing problem does not seem to shed any light.
Rather, we explain how the proof of convexity that we made in the minimum-time case fails when applied to
the minimum-energy case.

When $p_\omega=0$, the Hamiltonian in the minimum-energy case is given by \eqref{eq:30} and can be written as follows
\begin{equation}
  \label{eq:39}
 H=n^{-5/3}\left[2(3n\,p_{n})^{2}+5p_{\varphi}^{2}\right] 
\end{equation}
in the coordinates
$(n,\varphi,p_n,p_\varphi)$ that result from the symplectic change of coordinates $e=\sin\varphi$, $p_\varphi=\sqrt{1-e^{2}}p_e$.
The Hamiltonian equations can be written
$$
\begin{array}{ll}
  \dot n=12 \,n^{-2/3}\,(3n\,p_{n})\,, \hspace{2em}& {\textstyle\frac{\mathrm{d}}{\mathrm{d}t}} (3n\,p_{n}) =
  5n^{-5/3}\left[2(3n\,p_{n})^{2}+5p_{\varphi}^{2}\right]\,,
\\
\dot \varphi = 10\,n^{-5/3} \,p_\varphi\,, & \hspace{3em}\dot{p}_\varphi=0\,.
\end{array}
$$
With the same polar coordinates as in \eqref{eq:19}\; (namely $\rho\cos\psi=3\,n\,p_n$, $\rho\sin\psi=-\,p_\varphi$),\; and 
the time reparametrization $\mathrm{d}t=5\,n^{-5/3}\mathrm{d}\tau$,\; the state equations of these Hamiltonian equations have the form
\begin{equation}
  \label{eq:40}
  \mathrm{d}\psi/\mathrm{d}\tau=-\sin\psi\,(2+3\sin^2\!\psi)\,,\ \ 
\mathrm{d}\varphi/\mathrm{d}\tau=2\sin\psi\,.
\end{equation}

It is easy to describe the solutions of these equations on the cylinder $\mathcal{C}$ (see \eqref{eq:0069}).
There are two lines of equilibria at $\psi=0$ and $\psi=\pi$ and 
\begin{equation}
  \label{eq:41}
  \varphi+\sqrt{\frac25}\arctan\left(\sqrt{\frac52}\tan\psi\right)
\end{equation}
is a first integral (it is smooth at $\psi=\frac\pi2$).
These solutions are drawn on Figure~\ref{fig:L2}. It is clear
that, on a solution, the maximum possible variation of the variable $\varphi$ is $\sqrt{2/5}\,\pi$;\; this implies that, if
$|\varphi^0|>(\sqrt{2/5}-\frac12)\pi$,\; there are some values of $\varphi$ that cannot be reached by any solution
starting from the line $\{\varphi=\varphi^0\}$.

\section{Conclusion and open problems}

We have studied the average minimum time problem as described in section~\ref{sec:Tmin:Ham}.
This is a reduced subproblem of the planar transfer problem: the state has dimension 2, whereas it would have
dimension 3 in the real planar problem (we have set $p_\omega=0$; this imposes that the cyclic variable is constant
along transfers) and dimension 5 in the full problem where the plane containing the orbits is not fixed.

In \cite{Bonn-Cai09forum,Bonn-Cai-Duj06b}, the energy problem in full dimension is treated; the planar case is integrable (but
only the reduced planar case is flat); the full problem is not integrable but extremals may still be computed
explicitly.
Studying minimum time in higher dimension is an interesting program.

Concerning the reduced problem considered here, the main contribution of the paper is to prove geodesic convexity of the
elliptic domain (any two points in the domain may be joined by an extremal trajectory). On the one hand, it is not clear
that this result holds true in higher dimension, and on the other hand, in the present small dimension, optimality
and/or uniqueness of the extremal trajectories has not been studied.

Finally the singularities of the Hamiltonian have been investigated roughly, mostly to ensure existence of a
Hamiltonian \emph{flow}. It would be interesting to better understand their nature and their role, in particular the
singularities they cause on the balls of small radius for the metric.


\bigskip

\appendix

\section{Proof of Proposition~\ref{lem-check}}
\label{sec-proof-check}
Let us prove that the seven points in Proposition~\ref{lem-check} are satisfied by $a,b,c$ given by \eqref{eq:0026}
(full control case) and also by $a,b,c$ given by \eqref{A2} (tangential thrust case). 

\smallskip

\noindent\textit{\ref{ass-sym}. Symmetries.}
Equations \eqref{eq:22} and \eqref{eq:0046} imply
\[\widetilde{I}(\pi + \psi, \varphi, E) = - \widetilde{I}( \psi, \varphi, E),\quad 
\widetilde{I}(-\psi, -\varphi, \pi - E) = \widetilde{I}( \psi, \varphi, E)\]
while \eqref{eq:22t} implies
  \[\widetilde{J}(\psi + \pi, \varphi, E) = -\widetilde{J}(\psi, \varphi, E),\quad \widetilde{J}(-\psi, -\varphi, \pi -
  E) = \widetilde{J}(\psi, \varphi, E).\]
Substituting in \eqref{eq:26} and \eqref{eq:26t} yields, using
the change of variable $E\to\pi-E$ in the integrals $L(-\psi, -\varphi)$ and $M(-\psi, - \varphi)$, 
$$L(\pi + \psi, \varphi) \!=\! L(-\psi, -\varphi) \!=\! L(\psi, \varphi),\ 
M(\pi + \psi, \varphi) \!=\! M(-\psi, - \varphi) \!=\! M(\psi, \varphi).$$
This yields identities \eqref{eq:0036} with $a,b,c$ given either by \eqref{eq:0026} or by \eqref{A2}.

\smallskip

\noindent\textit{\ref{ass-cauchy}. Uniqueness of solutions.}
This follows from the classical Cauchy-Lipschitz theorem away from $\mathcal{S}$ (see Proposition~\ref{lem:unic0}). On
$\mathcal{S}$,
\\- in the full control case ($a,b$ given by \eqref{eq:0026}), as seen in \cite{Bomb-Pom13}, the regularity properties
\eqref{eq:46} of the right hand side of \eqref{eq:0024} guarantee the existence and uniqueness of solutions to the
Cauchy problem (Kamke uniqueness Theorem \cite[chap. \!III, Th. \!6.1]{Hart82}),
\\- in the tangential thrust case  ($a,b$ given by \eqref{A2}), the same argument does not apply but 
one may check that the derivative of
$\tan\psi-\frac12(\pm1+\sin\varphi)/\cos\varphi$ along $\dot\psi=a(\psi,\varphi)$, $\dot\varphi=b(\psi,\varphi)$ is
nonzero along the curve $\tan\psi=\frac12(\pm1+\sin\varphi)/\cos\varphi$, hence the vector field is 
transverse to $\mathcal{S}$ and this implies uniqueness of solutions starting
from a point in $\mathcal{S}$ (see e.g. \cite{Fili88}).

Continuity of $\Phi$ in \eqref{zeq:12}, is, according to \cite[chap. V, Theorem 2.1]{Hart82}, guaranteed by uniqueness of solutions and continuity of $a,b$.

\smallskip

\noindent\textit{\ref{ass-b}. Sign and zeroes of $b$.}

\noindent\textit{\ref{ass-b}.1. Full control case ($a,b$ given by \eqref{eq:0026}).}
On the one hand, one has
\begin{align}
  \label{eq:0075}
  b(0,\varphi) =& \frac{2\,\cos\varphi\sin\varphi}{\pi}\!\int_0^{\pi} \!\!
\frac{\cos^2\! E \,\mathrm{d}E }{ \sqrt{1-\sin^2\!\varphi\cos^2\! E }},
\ \;
b(-\pi,\varphi) =-b(0,\varphi) \,.
\end{align}
On the other hand, the derivative of $b(\psi, \varphi)$ with respect to $\psi$ is given by
  \label{lem:calculs}
\begin{align}
\label{eq:0077}
\frac{\partial b}{\partial\psi}(\psi,\varphi)
=&
\frac{\cos\psi}{\pi}\!\int_0^{\pi} \frac{(1-\sin^2\!\varphi\cos^2\! E)^{4}\,\sin^2\! E }{\widetilde{I}(\psi,\varphi,E)^{3/2}}\,\mathrm{d}E\,.
\end{align}
The integrals in both equations are positive. 
Hence for any fixed $\varphi\geq0$, $b(\psi, \varphi)$ is increasing with
respect to $\psi$ on $(-\frac\pi2,\frac\pi2)$ and decreasing on $(\frac\pi2,\frac{3\pi}2)$; 
according to \eqref{eq:0075}, it is positive in $(0,\frac\pi2]$ 
and negative in $(\pi,\frac{3\pi}2]$ (identified with $(-\pi,-\frac{\pi}2]$), 
hence it must vanish for a unique value of $\psi$ between $-\frac\pi2$ and 0, that we call $Z_b(\psi)$, thus defining $Z_b:[0,\frac\pi2)\to(-\frac\pi2,0]$.
It also vanishes for a unique value of $\psi$ between $\frac\pi2$ and $\pi$ that must be equal to $\pi+Z_b(\psi)$ according to \eqref{eq:0036}.
According to \eqref{eq:0075}, $Z_b(0)=0$ and $Z_b(0)<0$ if $\varphi>0$.

\begin{figure}[ht]
$\ $ \hfill
  \begin{minipage}[b]{0.25\linewidth}
    \caption{Numerical plot of the function \eqref{eq:32} on the interval $[0,1)$.  Obtained with Maple
      15.  \label{fig:bsmooth} }\vspace*{4\baselineskip}
  \end{minipage}\hfill
  \includegraphics[width=.4\linewidth]{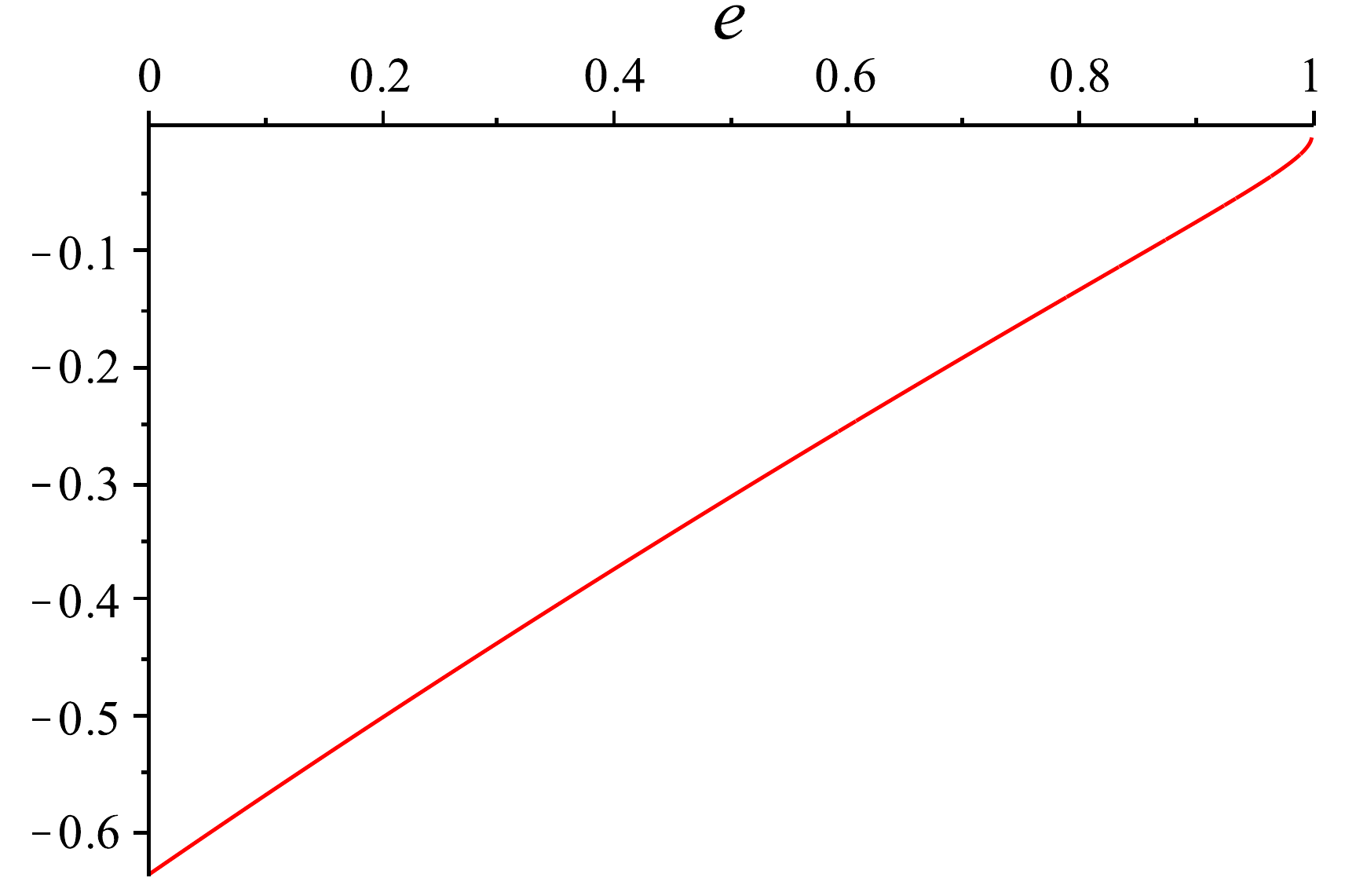}\hfill$\ $
\end{figure}

Proposition~\ref{lem:unic0} says that $L$, and hence $b$, are
smooth away from $\mathcal{S}$. The part of $\mathcal{S}$ that is contained in the square
$(-\frac\pi2,0]\times [0,\frac\pi2)$ is the curve
$\{\tan\psi\!=\!\frac{-1+\sin\varphi}{2\cos\varphi},\;0\!\leq\!\varphi\!<\!\frac\pi2\}$.
We claim that  $\gamma(\varphi)=b\left(\arctan\frac{-1+\sin\varphi}{2\cos\varphi}\,,\,\varphi\right)$ does not vanish between $0$ and
$\frac\pi2$; 
this is numerically checked by plotting, on Figure~\ref{fig:bsmooth}, the graph of the
function
$e\mapsto \gamma(\arcsin e)$ on $[0,1]$, i.e.
\begin{equation}
  \label{eq:32}
  e\mapsto b\left(\arctan\frac{-1+e}{2\sqrt{1-e^2}}\,,\,\arcsin e\right)\,.
\end{equation}
On the one hand, this proves that $b$ is smooth at points where it vanishes; on the other hand the derivative of $b$ with respect to $\psi$ is (see \eqref{eq:0077}) strictly positive at
$(Z_b(\varphi),\varphi)$, $\varphi>0$. This implies smoothness of $Z_b$ according to the inverse function theorem; and
this extends to negative $\varphi$ with $Z_b(0)=0$, hence point \ref{ass-b} of the proposition is satisfied with $\bar\sigma=0$;
it is also easy to check that $b(\psi,0)$ only if $\psi=0$ or $\psi=\pi$.




\smallskip

\noindent\textit{\ref{ass-b}.2. Tangential thrust case ($a,b$ given by \eqref{A2}).}
In the region $\mathcal{R}_1$, one has 
\begin{eqnarray}
b(\psi, \varphi) &=&
\left(\sign\cos\psi\right)\;\frac{2\cos\varphi\sin\varphi}\pi\int_{0}^{\pi}\frac{\cos^2\!E}{\sqrt{1-\sin^2\varphi
    \cos^2\!E}}\mathrm{d}E
\label{E2}\,.
\end{eqnarray}
The derivative of $b(\psi,\varphi)$ with respect to $\psi$ is zero in $\mathcal{R}_1$ because the above does not depend
on $\psi$, and in $\mathcal{R}_2$ it is given by
\begin{eqnarray}
b_{\psi}(\psi,\varphi)&=&
\frac{
\sign\bigl( P(\psi,\varphi)\bigr)\;
8\cos^2\!\varphi\;R(\psi,\varphi)\,\left(1-R(\psi,\varphi)\sin\varphi\right)
}{
\sqrt{1 - R(\psi,\varphi)^2}\sqrt{1 - R(\psi,\varphi)^2\sin^2\!\varphi }\,
(2\sin\varphi \cos\psi - \cos\varphi\sin\psi)^2}\,.
 \label{B1}
\end{eqnarray}
Since $|R|<1$ on $\mathcal{R}_2$, all factors are positive except $R(\psi,\varphi)$, hence
$b_{\psi}$ vanishes in $\mathcal{R}_2$ at points where $R$ vanishes, and this is exactly, according to \eqref{eq:35}, on
the lines $\{\psi = \pm \frac{\pi}{2}\}$, so that $b_{\psi}(\psi,\varphi)$ has the
sign of $P(\psi,\varphi)R (\psi,\varphi)$, i.e. of $\cos\psi$.
Hence, for fixed $\varphi\geq0$, $b(\psi,\varphi)$ is
\\- minimum and negative for $\psi=-\pi/2$,
\\- increasing for $\psi$ in $(-\frac\pi2,\arctan\left(\frac{-1+\sin\varphi}{2\cos\varphi}\right))$,
\\- constant, positive if $\varphi>0$ and zero if $\varphi=0$, for $\psi$ in $[\arctan\left(\frac{-1+\sin\varphi}{2\cos\varphi}\right),\arctan\left(\frac{1+\sin\varphi}{2\cos\varphi}\right)]$,
\\- increasing, hence positive, for $\psi$ in $(\arctan\left(\frac{1+\sin\varphi}{2\cos\varphi}\right),\frac\pi2)$,
\\- maximum and positive for $\psi=\frac\pi2$. 
\\
Hence, for any $\varphi>0$, there is a unique $\psi$, 
$-\frac\pi2<\psi<\arctan\left(\frac{-1+\sin\varphi}{2\cos\varphi}\right)$ such that $b(\psi,\varphi)=0$; we call it
$Z_b(\varphi)$, thus defining $Z_b:(0,\frac\pi2)\to(-\frac\pi2,0)$, satisfying
$-\frac\pi2<Z_b(\varphi)<\arctan\left(\frac{-1+\sin\varphi}{2\cos\varphi}\right)$ satisfying \eqref{zeq:7}-\eqref{zeq:9} (situation on $[\frac\pi2,\frac{3\pi}2]$ by symmetry, see \eqref{eq:0036}).

Since \eqref{E2} is valid also on
$\mathcal{S}$ by continuity, $b$ does not vanish on $\mathcal{S}$ except at $\varphi=0$, hence $b$ is smooth when it
vanishes, away from 
$\varphi=0$; since we also proved that $b_\psi$
is nonzero at these points, the inverse function theorem implies that $Z_b$ is smooth on the open interval
$(0,\frac\pi2)$; also the monotonicity argument shows that $\lim_{\varphi\to0}Z_b(\varphi)=-\arctan\frac12$, hence $Z_b$
is defined $[0,\frac\pi2)\to(-\frac\pi2,0)$ with $Z_b(0)=-\arctan\frac12$. Since our considerations above for
$\varphi\geq0$ imply that $b(0,\varphi)$ is zero if and only if
$\psi\in[-\arctan\frac12,\arctan\frac12]\cup[\pi-\arctan\frac12,\pi+\arctan\frac12]$, we have proved point \ref{ass-b} of the proposition with $\bar\sigma=\arctan\frac12$.



\smallskip

\noindent\textit{\ref{ass-a}. Sign and zeroes of $a$.} In \eqref{zeq:15}, the part saying that $a(\psi,0)>0$ for
$-\bar\sigma\leq\psi<0$ needs no proof in the full control case because $\bar\sigma=0$ and is easy in the tangential
thrust case because, from \eqref{eq:36} and \eqref{A2}, $a(\psi,0)=(1+\cos\psi)\sin\psi$ in $\mathcal{R}_2$.

We give numerical evidence
that $a(Z_b(\varphi),\varphi)$ is positive if $0<\varphi<\frac\pi2$. Note that the map $Z_b$ can
only be determined numerically as the zero $\psi=Z_b(\varphi)$ of $b(\psi,\varphi)=0$ between $-\frac\pi2$ and $0$ for fixed $\varphi$, but the determination 
is very reliable for $b$ is monotonous with respect to $\psi=Z_b(\varphi)$ in the considered region; see
point~\ref{ass-b}.

Figure~\ref{fig:sign-of-a} displays a numerical plot of the graph of the map $\varphi \mapsto a(Z_b(\varphi),\varphi)$ in the full control
case; we also show $\varphi \mapsto Z_b(\varphi)$. Figure~\ref{Fig_4} displays a numerical plot of the graph
of $\varphi \mapsto a(Z_b(\varphi),\varphi)$ in the tangential thrust case; we also show $\varphi \mapsto Z_b(\varphi)$
and $\varphi \mapsto \arctan\left(\frac{-1+\sin\varphi}{2\cos\varphi}\right)$ to show that it is very close to $Z_b$
($\psi=\arctan\left(\frac{-1+\sin\varphi}{2\cos\varphi}\right)$ is the curve where $R(\psi,\varphi)=-1$, the border
between $\mathcal{R}_1$ and $\mathcal{R}_2$).

\begin{figure}[ht]
$\ $ \hfill
  \begin{minipage}[b]{0.25\linewidth}
\caption{Plots (obtained with Matlab) of the maps $\varphi\mapsto Z_b(\varphi)$ (dashed line) and $\varphi \mapsto a(Z_b(\varphi),\varphi)$
      (solid line) in the full
      control case.  \label{fig:sign-of-a}   }\vspace*{4\baselineskip}
  \end{minipage}$\ $\hfill$\ $
 \includegraphics[width=.5\linewidth]{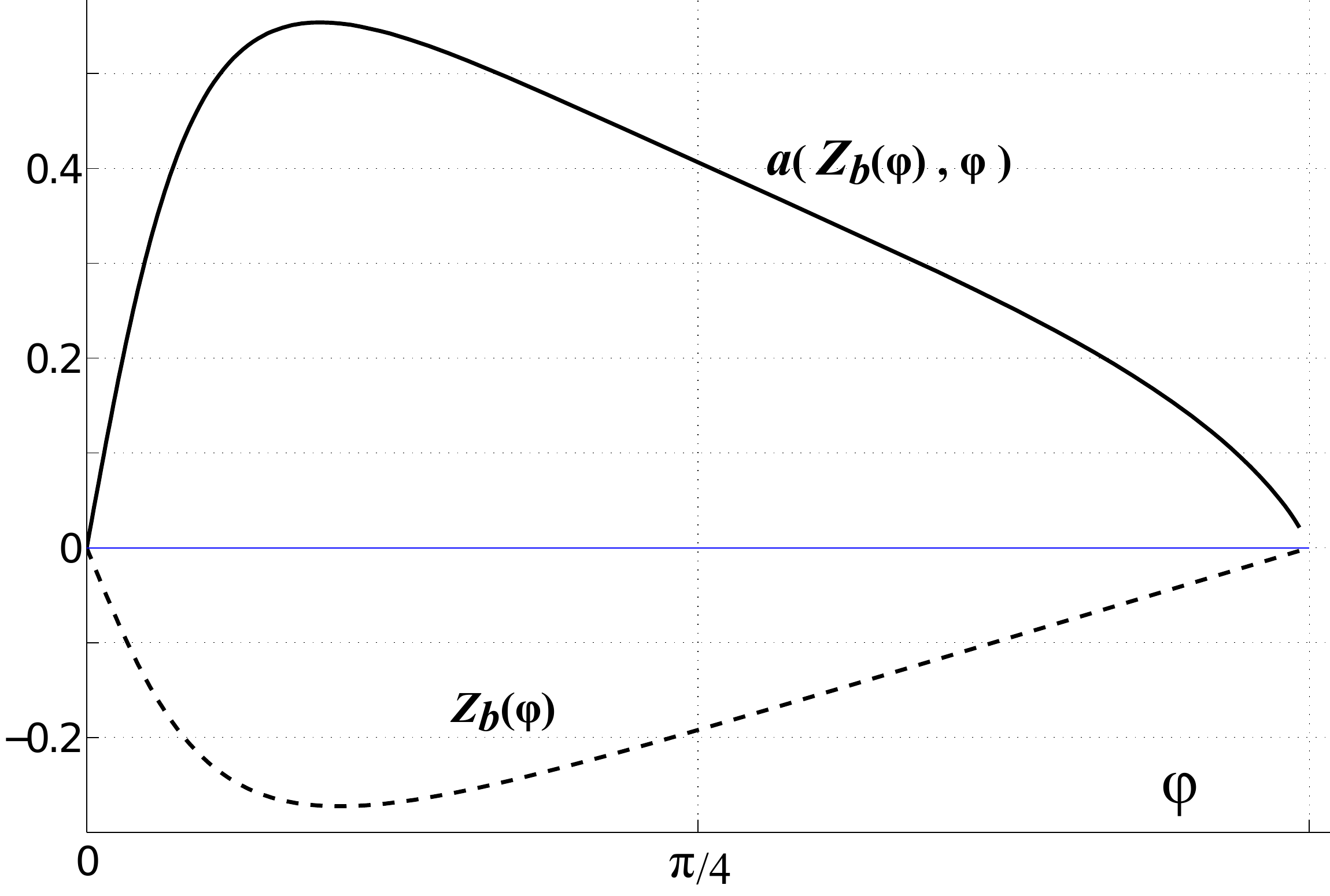}
\hfill$\ $
\end{figure}
\begin{figure}[ht]
$\ $ \hfill
  \begin{minipage}[b]{0.25\linewidth}
\caption{\label{Fig_4}The plots (obtained with Matlab) of the functions $\varphi\mapsto Z_b(\varphi)$ (dashed) and $\psi = a(Z_b(\varphi), \varphi)$ for $ \varphi \in[0, \pi/2]$. 
Note that $ a(Z_b(\varphi), \varphi)$ is everywhere positive on this interval. The other
curve shown is the curve $R(\psi, \varphi) = -1$.}\vspace*{6\baselineskip}
  \end{minipage}$\ $\hfill$\ $
 \includegraphics[width=.5\textwidth]{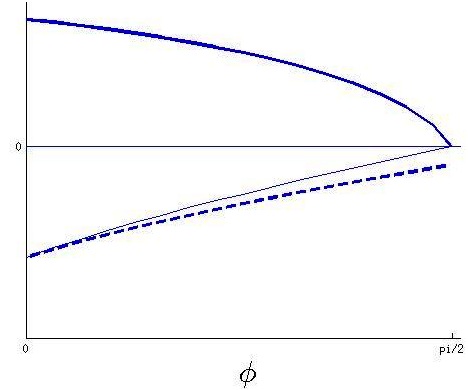}
\hfill$\ $
\end{figure}

\smallskip

\noindent\textit{\ref{ass-saddle}. Hyperbolic saddle.} Smoothness around the origin follows from
Proposition~\ref{lem:unic0}. It is clear in both cases that  $a(0,0) = 0$,  $b(0,0)= 0$.  The computation of the
Jacobians is easy (in the tangential thrust case it takes place in the region $\mathcal{R}_1$ with $\cos\psi \geq 0$,
see \eqref{E2}).

In the full control case,
\begin{equation}
  \label{eq:00107}
  \frac{\partial a}{\partial\psi}(0,0)=-2
\,,\ 
  \frac{\partial a}{\partial\varphi}(0,0)=\frac12
\,,\ 
  \frac{\partial b}{\partial\psi}(0,0)=\frac12
\,,\ 
  \frac{\partial b}{\partial\varphi}(0,0)=1
\,.
\end{equation}
The eigenvalues of the Jacobian $\bigl( \begin{smallmatrix}
  -2&1/2\\ 1/2 & 1
\end{smallmatrix} \bigr)$ are $(\sqrt{10}-1)/2$ (unstable) and $-(\sqrt{10}+1)/2$ (stable) 
associated to the eigenvectors $(\sqrt{10}-3,1)$ and $(-\sqrt{10}-3,1)$, respectively.

In the tangential thrust case,
\begin{equation}
\frac{\partial a}{\partial \psi}(0,0) = -2,\;\frac{\partial a}{\partial \varphi}(0,0) = \frac{1}{2},\;\frac{\partial b}{\partial \psi}(0,0) = 0,\;
 \frac{\partial b}{\partial \varphi}(0,0)  = 1.\label{eq:28}
\end{equation}

The eigenvalues of the Jacobian $\bigl( \begin{smallmatrix}
  -2&1/2\\ 0 & 1
\end{smallmatrix} \bigr)$ at $(0,0)$ are $1$ (unstable) and $-2$ (stable), associated to the eigenvectors $(1/6,1)$ and $(1,0)$, respectively.

\smallskip

\noindent\textit{\ref{ass-c}.  Values of $c$ at equilibria.} One deduces
$L(0,0)=L(\pi,0)=1$ from \eqref{eq:26}, \eqref{eq:22}, \eqref{eq:0046}  and
$M(0,0)=M(\pi,0)=1$ from \eqref{eq:26t}, \eqref{eq:22t}.
According to \eqref{eq:0026} and \eqref{A2}, this implies $c(0,0) = 1$, $c(0,\pi) = -1$ in both cases.

\smallskip

\noindent\textit{\ref{ass-US}. Stable and unstable manifolds of $(0,0)$.} 

\noindent\textit{\ref{ass-US}.1 full control case.}
The unstable manifold is the union of the equilibrium $(0,0)$ and two solutions that tend to $(0,0)$ as time $\tau$ tends to
$-\infty$ and are, according to \eqref{eq:00107}, both tangent to the line $\{\psi=(\sqrt{10}-3)\varphi\}$ at $(0,0)$.
One of the solutions approaches with positive $\psi$ and $\varphi$ and the other with negative $\psi$ and $\varphi$. We 
consider only the first one and call it $\tau\mapsto(\bar\psi(\tau),\bar\varphi(\tau))$, defined on the time interval
$(-\infty,\bar\tau^+)$, $\bar\tau^+\leq+\infty$; the result for the other one follows
by symmetry.  
Let $\mathcal{D}_1$ be the rectangle
$$\mathcal{D}_1=\{(\psi,\varphi)\in\mathcal{C},\;0<\psi<\pi,\,0<\varphi\}\,.$$
On the one hand, we have $\lim_{\tau\to-\infty}\bar\psi(\tau)=\lim_{\tau\to-\infty}\bar\varphi(\tau)=0$,
$\lim_{\tau\to-\infty}({\bar\psi}(\tau)/{\bar\varphi}(\tau))=\sqrt{10}-3$, hence $(\bar\psi(\tau),\bar\varphi(\tau))\in \mathcal{D}_1$ for $\tau$ close enough to $-\infty$.
On the other hand, the border of $\mathcal{D}_1$ is 
made of the two equilibria and three segments
$$
\{\psi=0,\,0<\varphi<\frac\pi2\}\,,\ \ 
\{\psi=\pi,\,0<\varphi<\frac\pi2\}\,,\ \ 
\{0<\psi<\pi,\, \varphi=0\}\,.
$$
A short computation shows that
\begin{equation}
\label{eq:0098}
a(0,\varphi) =\frac{\cos\varphi\sin\varphi}{\pi}\!\int_0^{\pi} \frac{\cos^2 E \,\mathrm{d}E }{
  \sqrt{1-\sin^2\varphi\cos^2 E }}\,,
\ \ \ a(\pi,\varphi)=-a(0,\varphi)\,,
\end{equation}
hence $a$ is positive on the first segment and negative on the second one; according to the proof of Point~\ref{ass-b}
above, $b$ is positive on the last one; hence solutions starting on these segments all enter $\mathcal{D}_1$. This proves
positive invariance of $\mathcal{D}_1$ (solutions may ``exit'' through the
segment $\{\varphi=\frac\pi2\}$, but they are no longer defined).
Hence the solution remains in $\mathcal{D}_1$ for all time in the open interval $(-\infty,\bar{\tau}^+)$.
According to Point~\ref{ass-b}, $b(\psi, \varphi)$ is positive on $\mathcal{D}_1$.
This solution cannot remain in a compact subset of $\mathcal{C}$ for all time because then it would have a non-empty
$\omega$-limit set that would have to be a union of equilibria and periodic solutions by Poincaré-Bendixon Theorem, but
the fact that $\dot\varphi>0$ in $\mathcal{D}_1$ prevents periodic solutions from existing and 
 the only equilibria are
$(0,0)$ and $(\pi,0)$, that cannot be approached because $\bar{\varphi}(t)$ cannot become small for $\bar\varphi(\tau)$ is increasing.
Hence necessarily, $\lim_{t\to \bar{\tau}^+}\bar\varphi(t)=\frac\pi2$. 
We have established that the parametrized curve $t\mapsto(\bar\psi(t),\bar\varphi(t))$, $-\infty<t<\bar{\tau}^+$ defines the
graph of a function $\psi= U(\varphi)$, $(0,\frac\pi2)\to(0,\pi)$; it is continuously differentiable from the implicit function theorem: since the
right-hand side of the differential equation is continuous, the parameterized curve is continuously differentiable, and
we saw that the derivative of $\varphi$ with respect to the parameter (time) remains positive (again because $b>0$ in $\mathcal{D}_1$).

Let us turn to the stable manifold. It is the union of the equilibrium $(0,0)$ and two solutions that tend to $(0,0)$ as time $\tau$ tends to
$+\infty$. According to the proof of point \ref{ass-saddle}, both solutions are tangent to the line $\{\psi=-(\sqrt{10}+3)\varphi\}$ at the origin.
We consider the solution that approaches $(0,0)$ with negative $\psi$ and positive $\varphi$, and call it $\tau\mapsto(\bar\psi(\tau),\bar\varphi(\tau))$; the result for the other one
follows by symmetry.
The proof is now very similar to the one for the unstable manifold, reversing time and replacing $\mathcal{D}_1$ by the domain
$$\mathcal{D}_2=\{(\psi,\varphi)\in\mathcal{C}\,,\;-\pi<\psi<Z_b(\varphi)\,,\ \varphi>0\}\,.$$

Firstly, $b(\psi, \varphi)$ is negative in this domain.
Secondly, the solution is in this domain for $\tau$ large enough (obviously $\psi$ is negative and $\varphi$ is positive, and it
is on the right side of $\psi=Z_b(\varphi)$ because $\varphi$ tends to zero so the solution must spend some
time in the region where $b<0$).
Thirdly, the domain $\mathcal{D}_2$ is negatively invariant: its border is made of the equilibria, the segments 
$\{\psi=-\pi,\,0<\varphi< \frac\pi2\}$, 
$\{-\pi<\psi<0,\, \varphi=0\}$ and the curve $\{\psi=Z_b(\varphi), 0<\varphi< \frac\pi2\}$. Solutions which start on the
segments leave $\mathcal{D}_2$ because of \eqref{eq:0098} and the fact that that $b(\psi,0)<0$ if $-\pi<\varphi<0$ (see
Point~\ref{ass-b} above). Solutions which start on the curve leave the domain $\mathcal{D}_2$ because of Point~\ref{ass-a}
above (at these points, $Z_b$ is differentiable, $Z_b(\bar\psi(\tau),\bar\varphi(\tau))=0$,
$a(\bar\psi(\tau),\bar\varphi(\tau))>0$, $b(\bar\psi(\tau),\bar\varphi(\tau))=0$).
This with the  second point implies that the solution is in the domain for all time.
The end of the proof, i.e. definition of the continuously differentiable $S$ is exactly the same as the previous proof, only with $b<0$  instead of
$b>0$. Moreover, we get for free that $-\pi<S(\varphi)<Z_b(\varphi)$ from the definition of $\mathcal{D}_2$; this and the above
implies \eqref{zeq:2}.

\smallskip
\noindent\textit{\ref{ass-US}.2 tangential thrust case.} Let us first compute some values of $a$ and $b$ on special
lines\footnote{
   The journal version unfortunately contains some misprints in equations (79)-(80); they are
   corrected here.
}: according to \eqref{A2} and \eqref{eq:26t}-\eqref{eq:22t},
  \begin{eqnarray}
 a\left(\frac{\pi}{2},\varphi\right) &=& -\frac{2\, \lvert\cos\varphi\rvert}{\sin\varphi}\;\ln\left(\frac{1+\sin\varphi}{1-\sin\varphi}\right)\;,\nonumber\\
 a\left(-\frac{\pi}{2},\varphi\right) &=& \hphantom{-} \frac{2\, \lvert\cos\varphi\rvert}{\sin\varphi}\;\ln\left(\frac{1+\sin\varphi}{1-\sin\varphi}\right)\;, \label{Eq1} \\
 a\left(0,\varphi\right) &=& \frac{\cos\varphi\sin\varphi}\pi\int_{0}^{\pi} \!\!\frac{\cos^2\!E\;\mathrm{d}E}{\sqrt{1-\sin^2\!\varphi \,\cos^2\!E}}\;, \nonumber
 \end{eqnarray}
\begin{equation}
\label{Eq2}
b(\psi, 0) =
\begin{cases}
 \displaystyle \sign(\sin\psi) \,\frac2 \pi \,\sqrt{4 - \cot^2\psi}&\text{if }\tan\psi\geq\frac12\;\text{(i.e. in $\mathcal{R}_{2}$)}\,,
\\
  \ \ \ 0 &\text{if }\tan\psi\leq\frac12\;\text{(i.e. in $\mathcal{R}_{1}$)}\,.
\end{cases}
\end{equation}

The unstable manifold comprises the equilibrium point $(0,0)$ and two solutions that tend towards it as time $\tau$ tends to $-\infty$, tangent, according to \eqref{eq:28},  to the line $\{\psi = \varphi/6\} $ at $(0,0)$. 
Thus either both $\varphi$ and  $\psi$ are positive as they approach $(0,0)$ or they are both negative.
Of the two solutions, we will only consider the one where
$\psi,\varphi$ are both positive, and call it $\tau \rightarrow (\bar{\psi}(\tau), \bar{\varphi}(\tau))$.
Obviously, $(\bar{\psi}(\tau), \bar{\varphi}(\tau))$ is in the rectangle
$$\mathcal{D}_1 = \{(\psi, \varphi)\in \mathcal{C}\;:\; 0 < \psi < \pi/2, \varphi > 0\}$$
for $\tau$ negative large enough.
From Point~\ref{ass-b}.2 above, $b(\psi,\varphi) > 0$ in the whole of $\mathcal{D}_1$. 
From \eqref{Eq1} and \eqref{Eq2}, $a( 0, \varphi) > 0$, $a( \pi/2,\varphi) < 0$, and $b(\psi, 0) \geq 0$ for $0< \psi <
\pi$, $\varphi > 0$, thus $\mathcal{D}_1$ is positively invariant.
The rest of the proof follows exactly
the same argument as for the case of the unstable manifold in Point~\ref{ass-US}.1 above.

We now consider the stable manifold. 
It comprises the equilibrium point $(0,0)$ and two solutions that tend towards it as time $\tau$ tends to $+\infty$.
They are both tangent  at $(0,0)$ to the stable eigenvector, i.e. (see \eqref{eq:28} and the sequel) to the line $\{\varphi = 0\}$.
Since $\dot\varphi$, i.e. $b(\psi,\varphi)$, is zero on the segment $\{\varphi=0,\,-\arctan\frac12\leq\psi\leq \arctan\frac12\}$, these 
solutions follow this segment.
We examine the one that approaches $(0,0)$ with negative $\psi$, the other one follows by symmetry. 
Call $\tau\mapsto (\bar{\psi}(\tau), \bar{\varphi}(\tau))$ the solution such that 
$\bar{\psi}(0)=-\arctan\frac12$, $\bar{\varphi}(0)=0$. One has $\bar{\varphi}(\tau)=0$ for all positive $\tau$ and
$\bar{\psi}(\tau)$ is increasing for positive $\tau$ and tends to zero as $\tau\to+\infty$.
Define the domain $$\mathcal{D}_2 = \{(\psi, \varphi)\in \mathcal{C}\;:\;-\frac\pi2 <
\psi < Z_b(0), \varphi > 0\}\,.$$ 
The solution is outside $\mathcal{D}_2$ for positive $\tau$ but on its border at $\tau=0$.
From \eqref{Eq1}, $a(-\frac\pi2,\varphi) > 0$ for all $\varphi$ between 0 and $\frac\pi2$; 
from \eqref{Eq1}, $b(\psi,0) < 0$ for $\psi \in (-\frac\pi2,-\arctan\frac12)$ ($-\arctan\frac12$ is $ Z_b(0))$);
from point \ref{ass-a} above, $a(Z_b(\varphi), \varphi) > 0$ if $\psi \in ( 0,\frac\pi2)$. Hence $\mathcal{D}_2$, as
well as its topological closure, are negatively invariant. Since $(\bar{\psi}(0), \bar{\varphi}(0))$ is on the boundary
of $\mathcal{D}_2$, one has $(\bar{\psi}(\tau), \bar{\varphi}(\tau))\in\mathcal{D}_2$ for all $\tau\in(\tau^-,0)$ where
$(\tau^-,+\infty)$ is the maximal interval of definition of the solution we consider.
From \eqref{zeq:9} and \eqref{zeq:7}, $b(\psi,\varphi)<0$ for all  $(\psi,\varphi)$ in $\mathcal{D}_2$.
Then, following the same argument as in the proof concerning the stable manifold in Point~\ref{ass-US}.1 above, we
obtain that the restriction to negative times of the solution is the graph $\varphi=S(\psi)$ where
$S:[0,\frac\pi2)\to(-\frac\pi2,0]$ is continuously differentiable on $(0,\frac\pi2)$ and $S(0)=-\arctan\frac12$.
We already noticed that the other part of the solution covers the segment $\{\varphi=0,\,-\arctan\frac12\leq\psi<0\}$.
This ends the proof of point Point~\ref{ass-US} (\eqref{zeq:1} to \eqref{zeq:2}) in the tangential thrust case.
\qed

\section{Proof of Theorem~\ref{prop:conv}}
\label{sec-proof-thm}

Theorem~\ref{prop:conv} and this section are independent of the rest of the paper: here we only refer to the
seven conditions ranging from equation \eqref{eq:0036} to equation \eqref{zeq:2}.

\begin{lemma}
\label{lem:sym} 
 Assume that $a,b,c$ satisfy \eqref{eq:0036} (i.e. point \ref{ass-sym}).

\noindent
  \textbf{If} $\tau\mapsto(\psi(\tau),\varphi(\tau))\in\mathcal{C}$ is a solution of \eqref{eq:0024} defined on the time interval
  $[0,\tfin]$, \textbf{then} $\tau\mapsto(\psi^\sharp(\tau),\varphi^\sharp(\tau))$ and
  $\tau\mapsto(\psi^+(\tau),\varphi^+(\tau))$ with
  \begin{equation}
    \label{eq:0082}
    \begin{array}{lcl}
      \psi^\sharp(\tau)=-\psi(\tau)\,,&\ \ \ &\psi^+(\tau)=\psi(\tfin\!-\!\tau)+\pi\,,
\\
      \varphi^\sharp(\tau)=-\varphi(\tau)\,,&\ \ \ &\varphi^+(\tau)=\varphi(\tfin\!-\!\tau)\,,
    \end{array}
  \end{equation}
are also solutions of \eqref{eq:0024} on the same time interval $[0,\tfin]$ and they satisfy
\begin{equation}
  \label{eq:0085}
  \int_0^{\tfin}c(\psi^\sharp(\tau),\varphi^\sharp(\tau))\mathrm{d}\tau=\int_0^{\tfin}c(\psi^+(\tau),\varphi^+(\tau))\mathrm{d}\tau=\int_0^{\tfin}c(\psi(\tau),\varphi(\tau))\mathrm{d}\tau\,.  
\end{equation}
\end{lemma}
\begin{proof}
  This is straightforward.\qed
\end{proof}
We also use the ``$\sharp$'' and ``$+$'' notation to denote the transformations in $\mathcal{C}$:
\begin{equation}
  \label{zeq:17}
  (\psi,\varphi)^\sharp=(-\psi,-\varphi)\,,\ \ (\psi,\varphi)^+=(\pi+\psi,\varphi)\,.
\end{equation}

Let us make further constructions and remarks on the conditions \eqref{eq:0036}-\eqref{zeq:2} before proceeding with the
proof \textit{per se}.

\paragraph{Stable and unstable manifolds.}
Equation \eqref{zeq:16} implies that the Jacobian of the vector field at the equilibrium $(0,0)$ has two real eigenvalues, of which one is positive and the other negative; i.e. $(0,0)$ is an
hyperbolic saddle (see e.g. \cite[section 8.3]{Hirs-Sma-Dev04}).
Thus it has a stable manifold $\mathcal{S}^0$ and an unstable manifold $\mathcal{U}^0$; these are curves passing through
$(0,0)$ tangent to the corresponding eigenvectors. Their existence is a consequence of \eqref{zeq:16} but 
point \ref{ass-US} assumes a more specific description.

\paragraph{The number $\overline{\sigma}$.} Everything may be stated in a much simpler if $\overline{\sigma}=0$: in
particular in Points ~\ref{ass-b} and ~\ref{ass-US}, $S$ and $Z_b$ may be continued into continuous even maps
  $(-\frac\pi2,\frac\pi2)\to\RR$, and, for instance, the equations of $\mathcal{S}^0$ and $\mathcal{U}^0$, instead of
  \eqref{zeq:10}, as
  $\{\psi=S(\varphi),-\frac\pi2<\varphi<\frac\pi2\}$ and $\{\psi=U(\varphi),-\frac\pi2<\varphi<\frac\pi2\}$.
We would have preferred this simpler formulation but we do not assume $\overline{\sigma}=0$ because the proof of
Theorem~\ref{Th1} in the tangential case uses Theorem~\ref{prop:conv} with a nonzero $\overline{\sigma}$ ($\overline{\sigma}=\arctan\frac12$).

However, in order to avoid considering positive and negative $\varphi$'s as different cases in \eqref{zeq:10}, \eqref{zeq:7} and
  \eqref{zeq:9}, we define the functions
          $U^0\!\!:[-\frac\pi2,\frac\pi2]\!\to\![-\pi,\pi]$,
          $S^0\!\!:[-\frac\pi2,\frac\pi2]\setminus\{0\}\!\to\![-\pi,\pi]$,
          $Z_b^0\!\!:[-\frac\pi2,\frac\pi2]\setminus\{0\}\!\to\!(-\frac\pi2,\frac\pi2)$
after $S$, $U$, $Z_b$;
these functions are odd and coincide with the former on $(0,\frac\pi2]$:
        \begin{gather}
        \label{zeq:4}
        S^0(-\varphi)=-S^0(\varphi) ,\ \ \ U^0(-\varphi)=-U^0(\varphi) ,\ \ \ Z_b^0(-\varphi)=-Z_b^0(\varphi)\,,
      \\
        \label{zeq:5}
         0<\varphi\leq\frac\pi2\ \ \Rightarrow\ \ 
       S^0(\varphi)=S(\varphi) ,\ U^0(\varphi)=U(\varphi) ,\ Z_b^0(\varphi)=Z_b(\varphi)\,.
        \end{gather}
  Then, the description \eqref{zeq:10} of the stable and unstable manifolds of $(0,0)$ may be replaced by 
        \begin{equation}
          \label{zeq:010}
          \begin{array}{l}
            \mathcal{S}^0=\{(\psi,\varphi)\in\mathcal{C},\;\varphi\neq0\mbox{ and }\psi=S^0(\varphi)\}\cup\,[-\overline{\sigma},\overline{\sigma}]\!\times\!\{0\}\,,
        \\
            \mathcal{U}^0=\{(\psi,\varphi)\in\mathcal{C},\;\psi=U^0(\varphi)\}\,.
          \end{array}
        \end{equation}
  and \eqref{zeq:7} may be replaced by
       \begin{displaymath}
           b(\psi,\varphi)=0\ \Leftrightarrow\ 
             \begin{cases}
               &\varphi=0\mbox{ and }\psi\in[-\overline{\sigma},\overline{\sigma}]\cup[\pi-\overline{\sigma},\pi+\overline{\sigma}]
               \\
               \mbox{or}&\varphi\neq0\mbox{ and }\psi\in\{Z_b^0(\varphi)\}\cup\{Z_b^\pi (\varphi)\}\;.
             \end{cases}
       \end{displaymath}

\paragraph{The equilibrium point $(\pi,0)$.}
  The ``$+$'' symmetry (see \eqref{eq:0082}) obviously maps $(0,0)$ to $(\pi,0)$, the stable manifold of $(0,0)$ to the
  unstable manifold of $(\pi,0)$, the unstable one to the stable one and the set of zeroes of $b$ to
  itself. Define $\mathcal{S}^\pi$, $\mathcal{U}^\pi$, $Z_b^\pi$ after $\mathcal{S}^0$, $\mathcal{U}^0$, $Z_b^0$ by
       \begin{equation}
         \label{zeq:6}
         S^\pi(\varphi)=\pi+U^0(\varphi),\ \ U^\pi(\varphi)=\pi+S^0(\varphi),\ \ Z_b^\pi(\varphi)=\pi+Z_b^0(\varphi)\,.
       \end{equation}
From \eqref{zeq:4}, \eqref{zeq:5} and \eqref{zeq:6}, the relation \eqref{zeq:2} translates into
  \begin{equation}
    \label{zeq:22}
    \begin{array}{rll}
      0<\varphi<\frac\pi2&\Rightarrow&
      \begin{cases}
        S^0(\varphi)<Z_b^0(\varphi)<0<U^0(\varphi)\,,\\U^\pi(\varphi)<Z_b^\pi(\varphi)<\pi<S^\pi(\varphi)\,,
      \end{cases}
      \\[2.8ex]
      -\frac\pi2<\varphi<0 &\Rightarrow&
      \begin{cases}
        U^0(\varphi) <0<Z_b^0(\varphi)<S^0(\varphi)\,,\\S^\pi(\varphi) <\pi<Z_b^\pi(\varphi)<U^\pi(\varphi)\,.
      \end{cases}
    \end{array}
  \end{equation}

The stable and unstable manifolds of  $(\pi,0)$ are:
       \begin{equation}
         \label{zeq:23}
         \begin{array}{l}
           \mathcal{S}^\pi=\{(\psi,\varphi)\in\mathcal{C},\;\psi=S^\pi(\varphi)\}\,,
       \\
           \mathcal{U}^\pi=\{(\psi,\varphi)\in\mathcal{C},\;\varphi\neq0\mbox{ and }\psi=U^\pi(\varphi)\}\cup\,[-\overline{\sigma},\overline{\sigma}]\!\times\!\{0\}\,.
         \end{array}
       \end{equation}
  Also, \eqref{zeq:7} and \eqref{zeq:9} become:
    \begin{gather}
      \label{zeq:07}
      b(\psi,\varphi)=0\ \Leftrightarrow\
      \begin{cases}
        &\varphi=0\mbox{ and }\psi\in[-\overline{\sigma},\overline{\sigma}]\cup[\pi-\overline{\sigma},\pi+\overline{\sigma}] \\
        \mbox{or}&\varphi\neq0\mbox{ and }\psi\in\{Z_b^0(\varphi)\}\cup\{Z_b^\pi (\varphi)\}\;,
      \end{cases} 
      \\ 
      \label{zeq:09}
      b(\psi,\varphi)>0 \mbox{ if }
      \begin{cases}
        & \varphi=0\mbox{ and } \overline{\sigma}<\psi<\pi+\overline{\sigma} \\
        \mbox{or}& \varphi\neq0\mbox{ and } Z_b^0(\varphi)<\psi<Z_b^\pi(\varphi)\,.
      \end{cases}
   \end{gather}

\paragraph{Invariant regions of $\mathcal{C}$.}
\begin{figure}[ht]
  \centering
  \includegraphics[width=1\textwidth]{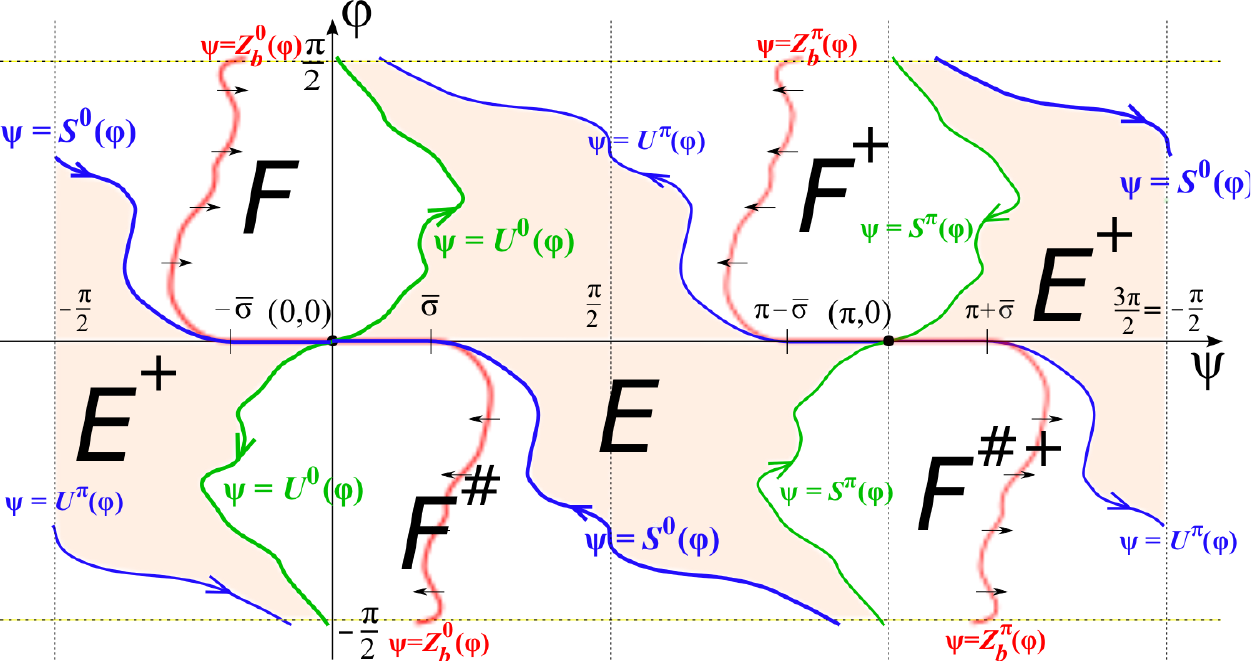}
  \caption{This picture reflects qualitatively the assumptions on $a$ and $b$. It is provided as a help to follow
    the proof; a precise numerical drawing is provided in Figure~\ref{Fig_1} and  Figure~\ref{Fig_2}  for
the specific expression of $a,b$ in the case of full control or tangential thrust.
\newline 
The six invariant regions separated by the stable and unstable manifolds of $(0,0)$ and $(\pi,0)$ are shown. The other
      curves are $\psi=Z_b^0(\varphi)$ and $\psi=Z_b^\pi(\varphi)$, where $b(\psi,\varphi)$ changes sign.
  \label{fig:regions}}
\end{figure}
The stable and unstable manifolds
$\mathcal{S}^0$, $\mathcal{U}^0$, $\mathcal{S}^\pi$,
$\mathcal{U}^\pi$, that intersect at the equilibria $(0,0)$ and $(\pi,0)$ are invariant sets that divide the cylinder $\mathcal{C}$ into six open regions:
  \begin{align}
\displaybreak[0]\label{zeq:100}
    F=&\,\{(\psi,\varphi),\;0<\varphi<\pi/2 \ \ \textup{and}\ \ S^0(\varphi) <\psi<U^0(\varphi)\}\,,\\
\displaybreak[0]\label{zeq:101}
    F^+=&\,\{(\psi,\varphi),\;0<\varphi<\pi/2 \ \ \textup{and}\ \ U^\pi(\varphi)<\psi<S^\pi(\varphi)\}\,,\\
\displaybreak[0]\label{zeq:102}
    F^\sharp=&\,\{(\psi,\varphi),\;-\pi/2<\varphi<0 \ \ \textup{and}\ \ U^0(\varphi)<\psi<S^0(\varphi)\}\,,\\
\displaybreak[0]\label{zeq:103}
    F^{\sharp+}=&\,\{(\psi,\varphi),\;-\pi/2<\varphi<0 \ \ \textup{and}\ \ S^\pi(\varphi)<\psi<U^\pi(\varphi)\}\,,\\
\displaybreak[0]\label{zeq:104}
 E=&\,
  \{(\psi,\varphi),\ S^0(\varphi)<\psi<S^\pi(\varphi) \;\textup{if}\;\varphi\leq0\,,\ \ U^0(\varphi)<\psi<U^\pi(\varphi)
  \; \textup{if}\;\varphi\geq0\,\}\,,
\\
\displaybreak[0]\label{zeq:105}
 E^+=&
\,
  \{(\psi,\varphi),\ U^\pi(\varphi)<\psi<U^0(\varphi) \; \textup{if}\; \varphi\leq0\,,\ \ S^\pi(\varphi)<\psi<S^0(\varphi) \; \textup{if}\; \varphi\geq0\,\}\,.
  \end{align}
$E$ is self-symmetric under the ``$\sharp$'' symmetry and $E^+$ is its own image under the ``+'' symmetry; $F^+$,
$F^\sharp$, $F^{\sharp+}$ are the images of $F$ by the ``$\sharp$'' and ``+'' symmetries; see Lemma~\ref{lem:sym}.
These regions are represented in Figure~\ref{fig:regions}.

\medskip

We now state and prove two preliminary lemmas and give the proper proof of Theorem~\ref{prop:conv}.

\begin{lemma}
\label{lem:regions}
Assume that $a,b$ satisfy assumptions \eqref{eq:0036} to \eqref{zeq:15}, and
consider a solution $t\mapsto(\psi(t),\varphi(t))$ of \eqref{eq:0024} defined on $[0,\tfin]$, $\tfin>0$.

\noindent 1.
If it starts in $E$ or in the upper part of the unstable manifold
$\mathcal{U}^0$, $\varphi(.)$ is monotonic increasing.

\noindent 2.
If it starts in $F$, then
\\- if it starts in 
$\{(\psi,\varphi),\,Z_b^0(\varphi)\leq\psi<U^0(\varphi)\}$, it remains in this part of $E$ and $\varphi(.)$ is monotonic increasing,
\\- if it starts in $\{(\psi,\varphi),\,S^0(\varphi)<\psi<Z_b^0(\varphi)\}$, either it remains in this part of $E$
and $\varphi(.)$ is monotonic decreasing, or there is some $\bar\tau$, $0<\bar\tau<\tfin$ such that
 $t\mapsto\varphi(t)$ is monotonic decreasing for $t$ between $0$ and $\bar\tau$, minimum for $t=\bar\tau$ and monotonic
 increasing for $t$ between $\bar\tau$ and $\tfin$, with $\psi(\bar\tau)-Z_b^0(\varphi(\bar\tau))=b(\psi(\bar\tau),\varphi(\bar\tau))=0$.

\noindent 3.
If  it starts in  the upper part of the stable manifold $\mathcal{S}^0$, $\varphi(.)$ is monotonic non-increasing.
\end{lemma}
\noindent
The behavior in the regions $E^+$, $F^+$, $F^\sharp$, $F^{\sharp+}$ and on the other pieces of stable or unstable curves
are obtained by symmetry; see Lemma~\ref{lem:sym}.
\begin{proof}
  Points 1 and 3 are obvious because $b$ is negative in $E$ and in the upper part of $\mathcal{U}^0$ while it is
  positive in the upper part of $\mathcal{S}^0$ except, if $\overline{\sigma}$ is nonzero, on the segment $\{\varphi=0\}$, where it is zero.
  Let us prove point 2.
  In the region $F$, according to \eqref{zeq:9}, $b$ has the sign of $\psi-Z_b^0(\varphi)$. Using differentiability of $Z_b^0$ 
  away from $\varphi=0$ (see \eqref{zeq:1}), one may compute the derivative of $\psi-Z_b^0(\varphi)$ with respect to time along a
  solution; it is is $a(\psi,\varphi)-{Z_b^0}'(\varphi)b(\psi,\varphi)$, which, according to \eqref{zeq:15}, is positive when
  $\psi-Z_b^0(\varphi)=0$, i.e. when $b(\psi,\varphi)=0$. Hence the region where $b>0$ is positively invariant, this 
  accounts for the behavior of solutions starting in $\{b>0\}$, 
  and no solution may stay on the locus where $b=0$, this accounts for solutions that start in $\{b<0\}$: either they
  stay in this part of $F$ or they cross $\{b=0\}$ at one time and then remain in $\{b>0\}$.\qed
\end{proof}

\smallskip

For any number $f$, $0<f<\frac\pi2$, let
\begin{equation}
  \label{zeq:37}
  \mathcal{C}_f=\{(\psi,\varphi)\in \mathcal{C}\,,\;|\varphi|<f\}.
\end{equation}

\begin{lemma}
  \label{lem:infini}
Assume that $a,b,c$ satisfy assumptions \eqref{eq:0036} to \eqref{zeq:18}.
For any number $f$, $0<f<\frac\pi2$, there are two neighborhoods $\Omega_0$ and $\Omega_\pi$ of $(0,0)$ and $(\pi,0)$ respectively and,
a number $T(f)\geq0$ such that
\begin{gather}
  \label{zeq:20}
  (\psi,\varphi)\in \Omega_0\Rightarrow c (\psi,\varphi)>{\textstyle\frac12}\,,\ \ 
  (\psi,\varphi)\in \Omega_\pi\Rightarrow c (\psi,\varphi)<-{\textstyle\frac12}\,,
\\
  \label{zeq:36}
  \mbox{no solution $t\mapsto(\psi(t),\varphi(t))$ of \eqref{eq:0024} may cross both $\Omega_0$ and $\Omega_\pi$,}
\end{gather}
and any solution $t\mapsto(\psi(t),\varphi(t))$ defined on the time
  interval $[0,\tfin]$ such that $(\psi(t),\varphi(t))\in\mathcal{C}_f$ for all $t\in[0,\tfin]$ satisfies
  \begin{equation}
    \label{zeq:40}
    \meas\,\{t\in[0,\tfin]\,,\ (\psi(t),\varphi(t))\notin\bigl(\Omega_0\cup \Omega_\pi\bigr)\}\leq T(f)\,.
  \end{equation}
\end{lemma}
The left-hand side ($\meas$ stands for the Lebesgue measure of a subset of $\RR$) is simply the time spent by the
solution outside the neighborhoods $\Omega_0$ and $\Omega_\pi$ of $(0,0)$ and $(\pi,0)$; this
bound depends on $f<\frac\pi2$ because $a$ and $b$ could tend to zero as $\varphi$ tends to $\pm\frac\pi2$.

\smallskip

\noindent\textit{Proof of Lemma~\ref{lem:infini}.}
To deal with the case where the stable or unstable manifolds contain a segment of the $\psi$-axis, we must account for both the 
cases where $\overline{\sigma} = 0$ and where $\overline{\sigma} \neq 0$.
We define the set $\mathcal{K}^{\varepsilon} $ ($\varepsilon>0$) as follows:
\begin{equation}
  \label{zeq:41}
  \mathcal{K}^{\varepsilon}=\varnothing \ \ \text{if}\ \ \overline{\sigma}=0
\end{equation}
and, if $\overline{\sigma}\neq0$,
\begin{equation}
  \label{zeq:44}
  \begin{array}{ll}
    \mathcal{K}^{\varepsilon}=&
      [\varepsilon,\overline{\sigma}+\varepsilon] \!\times\! [-\bar k \varepsilon,\bar k \varepsilon]
      \,\cup\,
      [-\overline{\sigma}-\varepsilon,-\varepsilon] \!\times\! [-\bar k \varepsilon,\bar k \varepsilon] \\
      & 
      \cup\;
      [\pi+\varepsilon,\pi+\overline{\sigma}+\varepsilon] \!\times\! [-\bar k \varepsilon,\bar k \varepsilon]
      \,\cup\,
      [\pi-\overline{\sigma}-\varepsilon,\pi-\varepsilon] \!\times\! [-\bar k \varepsilon,\bar k \varepsilon]
  \end{array}
\end{equation}
with
\begin{equation}
  \label{zeq:63}
  0<\bar k\leq\min\{\,1,\frac12 \left|\left. \frac{\partial a}{\partial\psi}(0,0)\right/
\frac{\partial a}{\partial\varphi}(0,0)\,\right|\,\}.
\end{equation}

Since $\frac{\partial b}{\partial\psi}(0,0)$ is zero,
\eqref{zeq:16}, \eqref{zeq:9} and \eqref{zeq:15} imply $\frac{\partial b}{\partial\varphi}(0,0)>0$, 
$\frac{\partial a}{\partial\psi}(0,0)<0$, and \eqref{zeq:2} implies $\frac{\partial a}{\partial\varphi}(0,0)\geq0$.
The slope of the tangent to the
curve $a=0$ is $-\left. \frac{\partial a}{\partial\psi}(0,0)\right/
\frac{\partial a}{\partial\varphi}(0,0)$, and the slope of the unstable manifold $\{\psi=U^0(\varphi)\}$ 
(i.e. the slope of the eigenvector corresponding to the negative eigenvalue $\frac{\partial a}{\partial\psi}(0,0)$) is 
$\left.\left( \partial b/\partial\varphi(0,0)-\frac{\partial a}{\partial\psi}(0,0)\right)\right/
\frac{\partial a}{\partial\varphi}(0,0)$, larger than the previous slope.
Hence, for some open ball $B$ around the origin,
\begin{equation}
\label{zeq:64}
|\varphi|\leq\bar k|\psi|,\, \psi\neq0,\, (\psi,\varphi)\in B 
\ \Rightarrow\
  |a(\psi,\varphi)|\neq0 \mbox{ and } |\psi-U^0(\varphi)|\neq0 \,.
\end{equation}
This implies that $a$ does not vanish on $B\cap([\varepsilon,\overline{\sigma}+\varepsilon]\!\times\![-\bar k
\varepsilon,\bar k \varepsilon])$.
On the compact segment $\{(\psi,0),\, 0\leq\psi\leq\overline{\sigma},\, (\psi,0)\notin B\}$, $|a(\psi,\varphi)|$ has a positive lower
bound $\underline{a}$; hence, for $\varepsilon$ small enough, $a (\psi,\varphi)$ is larger that $\frac12\underline{a}/2$ on the
part of the compact rectangle 
$[\varepsilon,\overline{\sigma}+\varepsilon]\times[-\bar k \varepsilon,\bar k \varepsilon]$ that is outside $B$. Gluing
the piece inside $B$ and the piece outside $B$ together, we get that, for $\varepsilon$ small enough, $a$ does not vanish on the
compact rectangle $[\varepsilon,\overline{\sigma}+\varepsilon]\times[-\bar k \varepsilon,\bar k \varepsilon]$. By
symmetry, i.e. from \eqref{eq:0036} and \eqref{zeq:44}, we get that, for $\varepsilon$ small enough,
\begin{equation}
  \label{zeq:47}
a \text{ does not vanish on }
\mathcal{K}^{\varepsilon}.
\end{equation}
By a similar argument, since $\psi-U^0(\varphi)$ does not vanish on the compact segment $\{(\psi,0),\,
0\leq\psi\leq\overline{\sigma},\, (\psi,0)\notin B\}$, \eqref{zeq:64} implies that it also does not vanish on
$\mathcal{K}^{\varepsilon}$ for $\varepsilon$ small enough either, hence
\begin{equation}
  \label{zeq:71}
  \mathcal{U}^0\cap \mathcal{K}^{\varepsilon}=\mathcal{U}^\pi\cap \mathcal{K}^{\varepsilon}=\varnothing
\end{equation}
for $\varepsilon$ small enough.
Define the neighborhoods $\Omega_0^{\varepsilon}$ and $\Omega_\pi^{\varepsilon}$
of $(0,0)$ and $(\pi,0)$ as:
\begin{gather}
    \label{zeq:32}
    \begin{split}
      \Omega_0^{\varepsilon}=&\; \{(\psi,\varphi),\; |\psi|<\varepsilon \,, |\varphi|<\varepsilon\}\,,\\
      \Omega_\pi^{\varepsilon}=\left(\Omega_0^{\varepsilon}\right)^+=&\; \{(\psi,\varphi),\; |\psi-\pi|<\varepsilon \,, |\varphi|<\varepsilon\}\,.
    \end{split}
  \end{gather}
Consider the two distinct solutions going through $(\frac\pi2,0)$ and $(-\frac\pi2,0)$; for ${\varepsilon}$ small enough they cross
neither $\Omega_\pi^{\varepsilon}$ nor $\Omega_\pi^0$, and hence they separate $\mathcal{C}$ into two regions, one containing
$\Omega_\pi^{\varepsilon}$ and the other one $\Omega_\pi^{\varepsilon}$. Hence, for ${\varepsilon}$ small enough, no
solution can cross both $\Omega_\pi^{\varepsilon}$ and $\Omega_0^{\varepsilon}$.
From \eqref{zeq:18}, it is also clear that, for ${\varepsilon}$ small enough, 
$$\min_{(\psi,\varphi)\in\Omega_0^{\varepsilon}}c (\psi,\varphi)>\frac12\ \ \ \text{and}\ \ \ \max_{(\psi,\varphi)\in\Omega_\pi^{\varepsilon}}c (\psi,\varphi)<-\frac12.$$
Let us now fix some $\varepsilon$ small enough that this is true, \eqref{zeq:47} and \eqref{zeq:71} hold and no
solution can cross both $\Omega_\pi^{\varepsilon}$ and $\Omega_0^{\varepsilon}$. Take
\begin{equation}
  \label{zeq:14}
  \Omega_0=\Omega_0^{\varepsilon},\  \Omega_\pi=\Omega_\pi^{\varepsilon},\ \mathcal{K}=\mathcal{K}^{\varepsilon}
\end{equation}
for this fixed value of $\varepsilon$. 
With this choice, one has
\begin{equation}
  \label{zeq:72}
  \mathcal{U}^0\cap \mathcal{K}=\mathcal{U}^\pi\cap \mathcal{K}=\varnothing\,,
\end{equation}
and there is some $\underline{a}>0$ such that
\begin{gather}
  \label{zeq:21}
  (\psi,\varphi)\in\ \mathcal{K}\;\Rightarrow|a(\psi,\varphi)|>\underline{a}\, 
\end{gather}
and  \eqref{zeq:20} and \eqref{zeq:36} hold: we only need to prove that \eqref{zeq:40} holds as well.

First, ``thicken'' the  curves $\psi=Z_b^0(\varphi)$ and $\psi=Z_b^\pi(\varphi)$ where $b$ vanishes, using the flow $\Phi$ (see \eqref{zeq:12}):
  \begin{equation}
     \label{zeq:33}
     \begin{split}
       \Sigma_0= &\; \{\Phi(Z_b^0(\varphi),\varphi,t) ,\;
            -{\textstyle\frac\pi2}\!<\! \varphi\!<\! {\textstyle\frac\pi2},\, \varphi\neq0,\, -{\textstyle\frac12}\!<\! t\!<\!{\textstyle\frac12} \}\cap \mathcal{C}_f,
\\
       \Sigma_\pi=(\Sigma_0)^+= &\; \{\Phi(Z_b^\pi(\varphi),\varphi,t) ,\;
            -{\textstyle\frac\pi2}\!<\! \varphi\!<\! {\textstyle\frac\pi2},\, \varphi\neq0,\, -{\textstyle\frac12}\!<\! t\!<\!{\textstyle\frac12} \}\cap \mathcal{C}_f
\,.
     \end{split}
   \end{equation}
\textit{Note:} if, for the initial condition $(\psi^o \!,\varphi^o)=(Z_b^\pi(\varphi),\varphi)$, either
$\tau^-\!\!>\!\!-\frac12$ or $\tau^+\!\!<\!\!\frac12$
(see \eqref{eq:0024}-\eqref{zeq:16}), then
$\Phi(Z_b^\pi(\varphi),\varphi,t)$ is not defined up to $-\frac12$ or $\frac12$; we however kept, for the sake of simplicity, ``$-{\textstyle\frac12}\!<\!t\!<\!{\textstyle\frac12}$'' instead of ``$\max\{-\frac12,\tau^-\}<t<\min\{\frac12,\tau^+\}$''.

The topological closure of 
$\mathcal{C}_f \setminus \bigl( 
\Omega_0\cup\Omega_\pi \cup \mathcal{K} \cup \Sigma_0\cup\Sigma_\pi
\bigr)$ does not contain any
zero of $b$, and is obviously compact. Hence $b$ has a positive lower bound $m(f)$ on that compact set:
\begin{equation}
  \label{zeq:43}
  (\psi,\varphi)\in
\mathcal{C}_f \setminus \bigl (\Omega_0\cup\Omega_\pi \cup \mathcal{K}\cup \Sigma_0\cup \Sigma_\pi\bigr)
\ \Rightarrow\ |b(\psi,\varphi)|>m(f)>0.
\end{equation}

Now consider a solution $[0,\tfin]\to\mathcal{C}_f$, and partition $[0,\tfin]$ as follows:
\begin{equation}
  \label{zeq:52}
\!
  \begin{array}{rl}
    [0,\tfin]=I_0\cup I_1\cup I_2\cup I_3\hspace{-7em}&
\\
\text{with}
   & I_0=\{t\in[0,\tfin]\,,\;
         (\psi(t),\varphi(t))\in \Omega_0\cup\Omega_\pi
      \,\}  \,,
\\
   &I_1= \{t\in[0,\tfin]\,,\;
          (\psi(t),\varphi(t)) \in \mathcal{K}
      \,\}  \,,
\\
   & I_2= \{t\in[0,\tfin]\,,\;
        (\psi(t),\varphi(t)) \in \Sigma_0\cup \Sigma_\pi \text{ and } t\notin I_0\cup I_1
      \,\}  \,,
\\
   & I_3= \{t\in[0,\tfin]\,,\;
         (\psi(t),\varphi(t))\notin  \Omega_0\cup\Omega_\pi \cup \mathcal{K}\cup \Sigma_0\cup \Sigma_\pi
      \,\}  \,.
  \end{array}  
\end{equation}
Obviously,
\begin{equation}
  \label{zeq:30}
\meas\,\{t\in[0,\tfin]\,,\ (\psi(t),\varphi(t))\notin\bigl(\Omega_0\cup \Omega_\pi\bigr)\}=
  \meas I_1 + \meas I_2 + \meas I_3 .
\end{equation}

Either the solution is one of the two equilibria
or it stays in one of the stable or unstable manifolds or in one of the six regions $E$, $E^+$, $F$, $F^+$, $F^\sharp$, $F^{\sharp+}$.
According to Lemma~\ref{lem:sym}, and seen that the neighborhoods are invariant by the $\sharp$ symmetry and exchanged by the $+$ symmetry,
 it is enough to prove the property for solutions in the regions $E$ and $F$, the equilibrium $(0,0)$ and
 the upper parts of its stable and unstable manifolds $\mathcal{S}^0$ and $\mathcal{U}^0$.

In order to bound $\meas I_1$ if $\overline{\sigma}\neq0$ (it is zero if $\overline{\sigma}=0$), let us prove that
\begin{equation}
  \label{zeq:78}
  \begin{array}{l}
    \text{either the solution does not cross $\mathcal{K}$ }
\\
    \text{or there are times $t_1,t_2$, $0\leq t_1<t_2\leq\tfin$ such that the solution is}
\\
    \text{in $\mathcal{K}$ on the time interval $[t_1,t_2]$ and outside $\mathcal{K}$ on $[0,\tfin]\setminus [t_1,t_2]$.} 
  \end{array}
\end{equation}
\begin{itemize}
\item This is obvious if the solution is an equilibrium or is on the unstable manifold $\mathcal{U}^0$, that do not
  cross $\mathcal{K}$ (see \eqref{zeq:72}).
\item In $E$: among the rectangles in \eqref{zeq:44}, only 
$[\varepsilon,\overline{\sigma}+\varepsilon]\!\times\! [-\bar k \varepsilon,\bar k \varepsilon]$ and
$[\pi-\overline{\sigma}-\varepsilon,\pi-\varepsilon] \!\times\! [-\bar k \varepsilon,\bar k \varepsilon]$ intersect $E$;
a solution that lies in $E$ cannot cross both rectangles because the solution passing through $(0,\frac\pi2)$ separates them.
Consider a solution that crosses one of them, say $[\varepsilon,\overline{\sigma}+\varepsilon]\!\times\! [-\bar k
\varepsilon,\bar k \varepsilon]$ (the situation around $[\pi-\overline{\sigma}-\varepsilon,\pi-\varepsilon] \!\times\!
[-\bar k \varepsilon,\bar k \varepsilon]$ is similar). Since $\dot\psi$ is negative in the rectangle ($a$ does not change sign according to
\eqref{zeq:47} and $a(0,\overline{\sigma})<0$ according to \eqref{zeq:15}) and $\dot{\varphi}$ is positive in $E$, 
a solution may only exit through the top or left-hand edge; if it exits through the top edge 
$\{\varphi=\bar k \varepsilon\}$, it will not enter again because $\varphi$ will remain larger than $\bar k
\varepsilon$. If it exits through the left-hand edge, the fact that $\dot{\varphi}>0$ in $E$ only allows it to enter again
through the same edge, but this is impossible because $\dot\psi>0$ on this edge; this proves \eqref{zeq:78}.
\item In $F$:
from \eqref{zeq:72}, $[-\overline{\sigma}-\varepsilon,-\varepsilon] \!\times\! [-\bar k \varepsilon,\bar k \varepsilon]$
is the only rectangle in \eqref{zeq:44} that intersects $F$; hence the solution may only cross this rectangle.
Since, according to \eqref{zeq:47}, $a$ does not change sign in the rectangle and, according to \eqref{zeq:15},
$a(0,\overline{\sigma})>0$, $a$ is positive on the rectangle, then the vector field points inwards on the left-hand edge
$\{\psi=-\overline{\sigma}-\varepsilon\}$ and outwards on the right-hand edge $\{\psi=-\varepsilon\}$. The bottom edge is not
in $F$.
A solution may only
exit through the top or right-hand edge; if it exits through the top edge, it means that $\varphi(t)$ is increasing at
the exit time, and so, according to Lemma~\ref{lem:regions}, it will continue increasing and cannot go back to the
rectangle. 
Also, if it exits through the right-hand
edge, re-entering the rectangle through the top or left-hand edge would require $\varphi$ to increase and reach at least
$k \varepsilon$, making it impossible to reach the rectangle again because $\varphi$ will continue increasing. This
proves \eqref{zeq:78} for  solutions that remain in $F$.
\item A solution in the upper
part of the stable manifold $\mathcal{S}^0$
also satisfies \eqref{zeq:78}  because it 
enters the rectangle through the  left-hand edge or the top edge and exits it through the right-hand edge and then goes
asymptotically to $(0,0)$.
\end{itemize}
We have proved \eqref{zeq:78} for any solution. Either $I_1 =\varnothing$ or $I_1=[t_1,t_2]$ and
connectedness implies that the solution stays in a single rectangle of $\mathcal{K}$. Using \eqref{zeq:21}, $\psi(t)$ varies monotonically
in $[t_1,t_2]$, and so its variation is at most $\overline{\sigma}$:
$\underline{a}(t_2-t_1)<\overline{\sigma}$. This yields
\begin{equation}
  \label{zeq:60}
  \meas I_1\leq \overline{\sigma}/\underline{a}\,.
\end{equation}

Solutions in $E$ or in the stable and unstable manifolds cross neither $\Sigma_0$ nor $\Sigma_\pi$; hence
$I_2=\varnothing$ for these solutions.
Solutions in  $F$ may cross $\Sigma_0$ but not $\Sigma_\pi$, and
they cannot enter  $\Sigma_0$ again after leaving it because the region between $\Sigma_0$ and the unstable
manifold $\mathcal{U}^0$ is invariant in positive time. They stay in $\Sigma_0$ on a time-interval of length at most 1
from the definition \eqref{zeq:33}.
Hence, for any solution,
\begin{equation}
  \label{zeq:61}
  \meas I_2\leq 1\,.
\end{equation}

Using Lemma \ref{lem:regions}, for any solutions in $E$ or $F$ or the upper part of $\mathcal{S}^0$ or
$\mathcal{U}^0$, the total variation of $\varphi$ on the interval $[0,\tfin]$ is at most $\pi-2\ell$. The inequality \eqref{zeq:43} then implies
\begin{equation}
  \label{zeq:62}
  \meas I_3\leq (\pi-2\ell) / m(f).
\end{equation}

Setting $T(f)= \overline{\sigma} /\underline{a} +1+ (\pi\!-\!2\ell) / m(f)$, 
\eqref{zeq:30}, \eqref{zeq:60}, \eqref{zeq:61} and \eqref{zeq:62}  imply \eqref{zeq:40}.\qed

\paragraph{Proof of Theorem~\ref{prop:conv}.}
  From Lemma~\ref{lem:sym},  the ``$+$'' symmetry allows one to interchange $\varphi^1$ and
$\varphi^0$ while the ``$\sharp$'' symmetry changes their sign:  we may assume
\begin{equation}
\label{eq:0028}
  -\frac\pi2<\varphi^0\leq\varphi^1<\frac\pi2\ \ \ \text{and}\ \ \ \varphi^1\geq0
\end{equation}
in the proof without loss of generality. We distinguish four cases.

\smallskip 

\textbf{\boldmath Case a: $-\frac\pi2< \varphi^0\leq0<\varphi^1<\frac\pi2$.}
In this paragraph, by convention,
\begin{equation}
  \label{zeq:66}
  \text{if $\varphi^0=0$, then $S^0(\varphi^0)$ stands for $\overline{\sigma}$ and $S^\pi(\varphi^0)$ stands for $\pi-\overline{\sigma}$.}
\end{equation}
The solutions such that $\varphi(0)=\varphi^0$, $\varphi(\tfin)=\varphi^1$, must be in the region $E$ (see
Lemma~\ref{lem:regions} and
Figure~\ref{fig:regions}), and satisfy $S^0(\varphi^0)<\psi(0)<S^\pi(\varphi^0)$.
For any $\chi$ in the open interval $(S^0(\varphi^0),S^\pi(\varphi^0))$, let
$(\psi^\chi(.),\varphi^\chi(.))$ be the unique solution
to the Cauchy problem \eqref{eq:0024} with initial condition
\begin{equation}
  \label{zeq:48}
  \psi^\chi(0)=\chi\,,\ \varphi^\chi(0)=\varphi^0\,.
\end{equation}
It is ---see \eqref{zeq:12}--- continuous with respect to $\chi$ and continuously differentiable with respect to
$\tau$.
Since this solution is in $E$, $\varphi^\chi$ is an increasing function of time, and so there is a unique time $\tfin^\chi>0$
such that
\begin{equation}
  \label{zeq:49}
  \varphi^\chi(\tfin^\chi)=\varphi^1\,.
\end{equation}
Since $b(\,\varphi^\chi(\tfin^\chi)\,,\,\psi^\chi(\tfin^\chi)\,)>0$, there is a constant $k>0$ such that
$|\varphi^\chi(t')-\varphi^\chi(t'')|>k|t'-t''|$ for $t',t''$ in a neighborhood of $\tfin^\chi$. This implies that
$\tfin^\chi$ depends continuously on $\chi$.
This allows us to define a continuous map $\Lambda:(S^0(\varphi^0),S^\pi(\varphi^0)) \to \RR$ by 
\begin{equation}
  \label{eq:00068}
  \Lambda(\chi)=\int_0^{\tfin^\chi}   c(\psi^\chi(s),\varphi^\chi(s))\mathrm{d}s.
\end{equation}
All we need to prove is that, for any $\bar\lambda\in\RR$, there is at least one $\chi$ in
$(S^0(\varphi^0),S^\pi(\varphi^0))$ such that $\Lambda(\chi)=\bar\lambda$, i.e. that $\Lambda$ is onto. Since $\Lambda$ is
continuous, it is sufficient to prove that
\begin{equation}
  \label{zeq:50}
  \lim_{\substack{\chi\to S^0(\varphi^0)\\\chi>S^0(\varphi^0)}}\Lambda(\chi)=+\infty
\,,\ \ 
\lim_{\substack{\chi\to S^\pi (\varphi^0)\\\chi< S^\pi(\varphi^0)}}\Lambda(\chi)=-\infty
\,.
\end{equation}
The solution  $t\mapsto(\psi^{S^0(\varphi^0)}(t),\varphi^{S^0(\varphi^0)}(t))$
of \eqref{eq:0024} with initial condition $(S^0(\varphi^0)$, $\varphi^0)$ is on the stable manifold of $(0,0)$; it is
defined on $[0,+\infty)$; $\varphi^{S^0(\varphi^0)}(t)$ is negative for all time and tends to zero as $t\to+\infty$.
By continuity with
respect to initial conditions, the solutions $(\psi^\chi(.),\varphi^\chi(.))$ starting from $(\chi,\varphi^0)$ with $\chi$ close enough to
$S^0(\varphi^0)$ are also defined on $[0,\overline{\tau}]$ for arbitrarily large fixed $\overline{\tau}>0$, and converge uniformly to
$(\psi^{S^0(\varphi^0)}(.),\varphi^{S^0(\varphi^0)}(.))$ on the compact interval $[0,\overline{\tau}]$ as $\chi\to S^0(\varphi^0)$.
This proves that, for $\chi$ close enough to $S^0(\varphi^0)$, $\varphi^\chi(\overline{\tau})<0$ and hence 
$\tfin>\overline{\tau}$.
The situation near $S^\pi (\varphi^0)$ being similar, we have proved that
\begin{equation}
  \label{zeq:45}
  \lim_{\substack{\chi\to S^0(\varphi^0)\\\chi>S^0(\varphi^0)}}\tfin^\chi=
\lim_{\substack{\chi\to S^\pi (\varphi^0)\\\chi< S^\pi(\varphi^0)}}\tfin^\chi=
+\infty\;.
\end{equation}
Now define $\Omega_0$ and $\Omega_\pi$ according to Lemma~\ref{lem:infini}, and $\overline{\tau}$ large enough that 
$(\psi^{S^0(\varphi^0)}(\overline{\tau}),$ $\varphi^{S^0(\varphi^0)}(\overline{\tau}))$ is in $\Omega_0$.
For $\chi$ close enough to
$S^0(\varphi^0)$, we also have
$(\psi^\chi(\overline{\tau}),\varphi^\chi(\overline{\tau})) \in S^0(\varphi^0)$;
hence, according to \eqref{zeq:36}, solutions $(\psi^\chi(.),\varphi^\chi(.))$ with $\chi$ close enough to
$S^0(\varphi^0)$ never cross $\Omega_\pi$. The interval
$[0,\tfin^\chi]$ is partitioned into times $t$ such that $(\psi^\chi(t),\varphi^\chi(t))\in\Omega_0$ and also
 $(\psi^\chi(t),\varphi^\chi(t))\notin\bigl(\Omega_0\cup \Omega_\pi\bigr)$. Since, according to \eqref{zeq:40} ( with
$f=\max\{|\varphi^0|,|\varphi^1|\}$) $\meas\{t\in[0,\tfin^\chi]\,,\ (\psi^\chi(t),\varphi^\chi(t))\notin\bigl(\Omega_0\cup \Omega_\pi\bigr)\}\leq T(f)$
and $\meas\,\{t\in[0,\tfin^\chi]\,,\
(\psi^\chi(t),\varphi^\chi(t))\in\Omega_0\} \geq\tfin^\chi-T(f)$, then \eqref{eq:00068} and \eqref{zeq:20} imply
\begin{equation}
  \label{zeq:53}
  \Lambda(\chi)\geq\,{\textstyle\frac12}\!\left(\tfin^\chi -T(f)\right)
-c_f\,T(f)
\ \ \ 
  \text{with}\ \ \ \ c_f=\!\!\max_{(\psi,\varphi) \in\mathcal{C},\,|\varphi|\leq f}   |c (\psi,\varphi)|\,,
\end{equation}
and this does imply, using \eqref{zeq:45}, the first limit in \eqref{zeq:50}. Similarly, for $\chi$ close enough to
$S^\pi(\varphi^0)$,
one has $\meas \{t\in[0,\tfin^\chi]\,,\ (\psi^\chi(t),\varphi^\chi(t))\in\Omega_\pi\} \geq\tfin^\chi-T(f)$,
$\meas\,\{t\in[0,\tfin^\chi]\,,\ (\psi^\chi(t),\varphi^\chi(t))\notin\bigl(\Omega_0\cup \Omega_\pi\bigr)\}\leq T(f)$ and
hence, using \eqref{zeq:53} again,
\begin{equation}
  \label{zeq:46}
  \Lambda(\chi)\leq -{\textstyle\frac12}\left(\tfin^\chi -T(f)\right) +
c_f\,\,T(f)
\end{equation}
with $c_f$ as in \eqref{eq:47}. This implies, according to \eqref{zeq:45}, the second limit in \eqref{zeq:50}.

\smallskip 

\textbf{\boldmath Case b: $-\frac\pi2< \varphi^0<0$ and $\varphi^1=0$.}
\underline{If $\overline{\sigma}=0$}, the solutions such that $\varphi(0)=\varphi^0$, $\varphi(\tfin)=\varphi^1$ must be
in the region $E$, and the proof from case (a) applies, where $\varphi^1$ is replaced with zero.
\underline {If $\overline{\sigma}>0$}, the solutions on the stable manifolds $\{\psi=S^0(\varphi)\}$ and $\{\psi=S^\pi(\varphi)\}$ (see
Figure~\ref{fig:regions}) also qualify for this case, because $\varphi$ reaches zero in finite time. Hence we have to examine the
solutions such that $(\psi(0),\varphi(0))=(\chi,\varphi^0)$ with $\chi\in [S^0(\varphi^0),S^\pi(\varphi^0)]$ instead of
the \emph{open} interval; the solutions $(\psi^\chi(.),\varphi^\chi(.))$ to the Cauchy problem \eqref{eq:0024}-\eqref{zeq:48}
are still well defined and depend continuously on
$\chi$; however, uniqueness of $\tfin^\chi$ such that \eqref{zeq:49}  holds for
$\chi\in(S^0(\varphi^0),S^\pi(\varphi^0))$ but not for $\chi=S^0(\varphi^0)$ or $\chi=S^\pi(\varphi^0)$. If we call
$\tfin^{S^0(\varphi^0)}$ (resp. $\tfin^{S^\pi(\varphi^0)}$) the \emph{first} time $t$ such that
$\varphi^{S^0(\varphi^0)}(t)=0$ (resp. $\varphi^{S^\pi(\varphi^0)}(t)=0$), the solutions to be considered are these with
initial condition $(\psi(0),\varphi(0))=(\chi,\varphi^0)$, $S^0(\varphi^0)<\chi<S^\pi(\varphi^0)$ on the time interval
$[0,\tfin^\chi]$ and these with initial condition $(\psi(0),\varphi(0))=(\chi,\varphi^0)$,
$\chi\in\{S^0(\varphi^0),S^\pi(\varphi^0)\}$ on the time intervals $[0,\tau]$, $\tfin^\chi\leq\tau<+\infty$.
With the first set of solutions, one reaches $\bar\lambda\in[\Lambda(S^0(\varphi^0)),\Lambda(S^\pi(\varphi^0))]$; the set of
solutions with initial condition $(S^0(\varphi^0),\varphi^0)$ allows one to reach $\bar\lambda$ larger than
$\Lambda(S^0(\varphi^0))$ as the value of
$
\int_0^{\tau}   c(\psi^{S^0(\varphi^0)}(s),\varphi^{S^0(\varphi^0)}(s)) \mathrm{d}s
$
varies from $\Lambda(S^0(\varphi^0))$ to $+\infty$
as $\tau$ varies from $\tau^{S^0(\varphi^0)}$ (the smallest such value such that $\varphi {S^0(\varphi^0)}(\tau)=0$) to
$+\infty$; these solutions with initial condition $(S^\pi(\varphi^0),\varphi^0)$ allow one to reach $\bar\lambda$ smaller than $\Lambda(S^\pi(\varphi^0))$.

\textbf{\boldmath Case c: $\varphi^0=\varphi^1=0$.}
It suffices to chose $\tfin=|\bar\lambda|$ and the solution to be the
equilibrium $(0,0)$ if $\bar\lambda\geq0$ or the equilibrium $(\pi,0)$ if $\bar\lambda\geq0$. Then \eqref{eq:0027} is
satisfied because $c(0,0)=1$, $c(\pi,0)=-1$. 

\smallskip 

\textbf{\boldmath Case d: $0< \varphi^0<\varphi^1<\frac\pi2$.}
This is the other ``generic'' case, with case (a).
According to Lemma \ref{lem:regions},
the solutions such that $\varphi(0)=\varphi^0$, $\varphi(\tfin)=\varphi^1$ must lie in one of the regions
$F$, $E$, or $F^+$
or in the unstable manifolds $\mathcal{U}^0$ or  $\mathcal{U}^\pi$ that separate them (see Figure~\ref{fig:regions}).
These solutions satisfy $\varphi(0)=\varphi^0$, hence $S^0(\varphi^0)<\psi(0)<S^\pi(\varphi^0)$. For any
$\chi$ in the open interval $(S^0(\varphi^0),S^\pi(\varphi^0))$, let
$(\psi^\chi(.),\varphi^\chi(.))$ be the solution  to the Cauchy problem \eqref{eq:0024}-\eqref{zeq:48}. 
According to Lemma \ref{lem:regions}:
\\- If $S^0(\varphi^0)<\chi< Z_b^0(\varphi^0)$ or
$Z_b^0(\varphi^\pi)<\chi<S^\pi(\varphi^0)$, $t\mapsto\varphi^\chi(t)$ is first decreasing, then crosses the set of
zeroes of $b$ at some time $t^o$: $\psi^\chi(t^o)=Z_b^0(\varphi^\chi(t^o))$ or
$\psi^\chi(t^o)=Z_b^\pi(\varphi^\chi(t^o))$ and is increasing for $t$ larger than $t^o$. Hence, since
$\varphi^\chi(t^o)<\varphi^0<\varphi^1$, there is a unique $\tfin^\chi$ (larger than $t^o$) such that
$\varphi^\chi(\tfin^\chi)=\varphi^1$. 
\\- If $Z_b^0(\varphi^0)\leq \chi\leq Z_b^\pi(\varphi^0)$,
$t\mapsto\varphi^\chi(t)$ is monotonic increasing for positive times and cannot have a limit, hence it takes all the
values between $\varphi^0$ and $\frac\pi2$ only once and there is a unique $\tfin^\chi$ such that \eqref{zeq:49} holds.
\\
In both cases, $Z_b^0(\varphi^1)<\psi^\chi(\tfin^\chi)<Z_b^\pi(\varphi^1)$, hence 
$b(\,\widehat\varphi^\chi(\tfin^\chi)\,,\,\widehat\psi^\chi(\tfin^\chi)\,)>0$, and so
there is a constant $k>0$ such that
$|\varphi^\chi(t')-\varphi^\chi(t'')|>k|t'-t''|$ for $t',t''$ in a neighborhood of $\tfin^\chi$. This implies that
$\tfin^\chi$ depends continuously on $\chi$. This continuous dependence on $\chi$ allows us to define the continuous map
$\Lambda:(U^0(\varphi^0),U^\pi(\varphi^0)) \to \RR$ by \eqref{eq:00068}; as in case (a), let us prove the following
limits, sufficient to imply that $\Lambda$ is onto:
\begin{equation}
  \label{zeq:51}
  \lim_{\substack{\chi\to U^0(\varphi^0)\\\chi>U^0(\varphi^0)}}\Lambda(\chi)=+\infty
\,,\ \ 
\lim_{\substack{\chi\to U^\pi (\varphi^0)\\\chi< U^\pi(\varphi^0)}}\Lambda(\chi)=-\infty
\,.
\end{equation}
This follows as in case (a): first we get
\begin{displaymath}
  \lim_{\substack{\chi\to U^0(\varphi^0)\\\chi>S^0(\varphi^0)}}\tfin^\chi=
\lim_{\substack{\chi\to U^\pi (\varphi^0)\\\chi< S^\pi(\varphi^0)}}\tfin^\chi=
+\infty
\end{displaymath}
and then, with $f=|\varphi_1|$,  \eqref{zeq:53} holds for $\chi$ close to $U^0(\varphi^0)$ and \eqref{zeq:46} for $\chi$ close to $U^\pi(\varphi^0)$.

\smallskip

\textbf{\boldmath Case e: $0< \varphi^0=\varphi^1<\frac\pi2$.}
This is a similar to the previous case but degenerate in the sense that $\tfin^\chi=0$ if
$Z_b^0(\varphi^0)\leq \chi\leq Z_b^\pi(\varphi^0)$.
The only nontrivial trajectories that display the
same initial and final values of $\varphi$ lie in the regions $F$ or $F^+$, and they join points on one side of the curve
where $b$ vanishes ($\psi=Z_b^0(\varphi)$ or $\psi=Z_b^\pi(\varphi)$) to points on the other side. We have
$$
\Lambda(\chi)
\begin{cases}
  >0&\text{if}\;U^0(\varphi^1)<\chi<Z_b^0(\varphi^1)\,,
\\
  =0&\text{if}\;Z_b^0(\varphi^1)\leq\chi\leq Z_b^\pi(\varphi^1)\,,
\\
  <0&\text{if}\;Z_b^\pi(\varphi^1)<\chi<U^\pi(\varphi^1)\,;
\end{cases}
$$
\eqref{zeq:51} still holds and the same arguments prove that $\Lambda$ is onto.
\qed


\end{document}